\documentclass[twoside, 11pt, reqno]{amsart}

\usepackage{amsfonts}
\usepackage[utf8]{inputenc}
\usepackage{array}

\usepackage[english]{babel}
\usepackage{latexsym}
\usepackage{amscd}
\usepackage{amsmath}
\usepackage{amssymb}
\usepackage[mathcal]{euscript}
\usepackage{amsthm}
\usepackage{graphicx}
\usepackage{mathrsfs}
\usepackage{mathtools}
\usepackage{pict2e}
\usepackage{stackrel} 
\usepackage{wasysym}
\usepackage{leftindex}
\usepackage[all]{xy}
\usepackage[dvipsnames]{xcolor}
\usepackage[colorlinks,final,hyperindex]{hyperref}
\usepackage[noabbrev,capitalize]{cleveref}
\usepackage{tikz}
\usepackage{tikz-cd}
\usetikzlibrary{decorations.pathmorphing,decorations.markings,arrows,calc,shapes.geometric}%
\usepackage{geometry}

\allowdisplaybreaks

\newtheorem{theorem}{Theorem}[section]
\newtheorem*{theorem*}{Theorem}
\newtheorem{lemma}[theorem]{Lemma}
\newtheorem{corollary}[theorem]{Corollary}
\newtheorem{proposition}[theorem]{Proposition}

\theoremstyle{definition}
\newtheorem{definition}[theorem]{Definition}
\newtheorem*{definition*}{Definition}
\newtheorem{remark}[theorem]{\sc Remark}
\newtheorem{example}[theorem]{\sc Example}

\newtheorem{construction}[theorem]{\sc Construction}
\newtheorem{assumptions}{\sc Assumptions}

\newcommand{\nocontentsline}[3]{}
\newcommand{\tocless}[2]{\bgroup\let\addcontentsline=\nocontentsline#1{#2}\egroup}

\hypersetup{citecolor=BleuMinuit, linkcolor=bordeau, filecolor=black, urlcolor=BleuMinuit}

\title{Properadic coformality of spheres}
\author{Coline Emprin}
\address{Coline Emprin, Department of Mathematics, Stockholm University, Albano hus 1, 106 91 \indent Stockholm, Sweden}
\email{\noindent \href{email:coline.emprin@math.su.se}{coline.emprin@math.su.se}}

\author{Alex Takeda}
\address{Alex Takeda, Department of Mathematics, Uppsala University,
Regementsv\"agen 10, 752 37 \indent Uppsala, Sweden} 
\email{\href{mailto:alex.takeda@math.uu.se}{alex.takeda@math.uu.se}}

\date{\today}



\definecolor{red}{RGB}{230,97,0}
\definecolor{blue}{RGB}{93,58,155}
\definecolor{Chocolat}{rgb}{0.36, 0.2, 0.09}
\definecolor{BleuTresFonce}{rgb}{0.215, 0.215, 0.36}
\definecolor{BleuMinuit}{RGB}{0, 51, 102}
\definecolor{bordeau}{rgb}{0.5,0,0}
\definecolor{turquoise}{RGB}{6, 62, 62}
\definecolor{rose}{RGB}{235, 62, 124}
\definecolor{pink}{RGB}{225, 219, 240}

\newcommand{\g}{\ensuremath{\mathfrak{g}}}

\newcommand{\antishriek}{\text{\raisebox{\depth}{\textexclamdown}}}

\newcommand{\h}{\ensuremath{\mathfrak{h}}}

\newcommand{\QQ}{\ensuremath{\mathbb{Q}}}

\newcommand{\C}{\ensuremath{\mathcal{C}}}

\newcommand{\End}{\ensuremath{\mathrm{End}}}

\newcommand{\Hom}{\ensuremath{\mathrm{Hom}}}
\renewcommand{\S}{\ensuremath{\mathbb{S}}}

\newcommand{\id}{\ensuremath{\mathrm{id}}}

\newcommand{\Z}{\mathbb{Z}}

\makeatletter

\newcommand{\R}{\mathbb{R}}

\newcommand{\calC}{\mathcal{C}}

\newcommand{\calF}{\mathcal{F}}
\newcommand{\calG}{\mathcal{G}}

\newcommand{\calL}{\mathcal{L}}

\newcommand{\calP}{\mathcal{P}}
\newcommand{\calQ}{\mathcal{Q}}

\newcommand{\calT}{\mathcal{T}}

\newcommand{\calV}{\mathcal{V}}

\newcommand{\calY}{\mathcal{Y}}

\newcommand{\Imm}{\mathop{\mathrm{Im}}\nolimits}

\newcommand{\wt}{\mathop{\mathrm{wt}}\nolimits}

\newcommand{\gr}{\mathop{\mathrm{gr}}\nolimits}

\newcommand{\verteq}{\rotatebox{90}{$\;\;=\;\;$}}

\DeclareMathSymbol{\antishrieksymbol}{\mathord}{operators}{'74}
\makeatletter
	\DeclareRobustCommand{\antishriek}{{\mathpalette\anti@shriek\relax}}
	\newcommand\anti@shriek[2]{%
		\raisebox{\depth}{$\m@th#1\antishrieksymbol$}%
	}
\makeatother
\newcommand{\antish}{\antishriek}

\tikzset{>={Stealth[scale=1.2]}}
\tikzset{->-/.style={decoration={
			markings,
			mark=at position #1 with {\arrow{>}}},postaction={decorate}}}
\tikzset{-w-/.style={decoration={
			markings,
			mark=at position #1 with {\arrow{Stealth[fill=white,scale=1.4]}}},postaction={decorate}}}
\tikzset{->-/.default=0.65}
\tikzset{-w-/.default=0.65}
\tikzstyle{bullet}=[circle,fill=black,inner sep=0.5mm]
\tikzstyle{circ}=[circle,draw=black,fill=white,inner sep=0.5mm]
\tikzstyle{vertex}=[circle,draw=black,thick,inner sep=0.5mm]
\tikzstyle{dot}=[draw,circle,fill=black,minimum size=0.5mm,inner sep = 0mm, outer sep = 0mm]
\tikzset{darrow/.style={double distance = 4pt,>={Implies},->},
	darrowthin/.style={double equal sign distance,>={Implies},->},
	tarrow/.style={-,preaction={draw,darrow}},
	qarrow/.style={preaction={draw,darrow,shorten >=0pt},shorten >=1pt,-,double,double
		distance=0.2pt}}

\newcommand{\tikzfig}[1]{\begin{tikzpicture}[auto,baseline={([yshift=-.5ex]current bounding box.center)}]#1\end{tikzpicture}}
\usetikzlibrary{decorations.pathreplacing}

\begin{document}

\begin{abstract}
    We define a properad $\calY^{(n)}_\infty$ that encodes $n$-pre-Calabi--Yau algebras with vanishing copairing. These algebras include chains on the based loop space of any space $X$ endowed with a fundamental class $[X]$ such that $(X,[X])$ satisfies Poincaré duality of degree $n \geqslant 1$ with local system coefficients, such as an oriented manifold. Extending the notion of coformality of spaces, we define coformality of such a pair $(X,[X])$ in terms of properadic formality of $\calY^{(n)}_\infty$-algebra structures on $C_*(\Omega X)$. Using a refined version of properadic Kaledin classes, we establish the intrinsic coformality of all spheres in characteristic zero.
\end{abstract}
	
\maketitle

\makeatletter
\def\@tocline#1#2#3#4#5#6#7{\relax
	\ifnum #1>\c@tocdepth %
	\else
	\par \addpenalty\@secpenalty\addvspace{#2}%
	\begingroup \hyphenpenalty\@M
	\@ifempty{#4}{%
		\@tempdima\csname r@tocindent\number#1\endcsname\relax
	}{%
		\@tempdima#4\relax
	}%
	\parindent\z@ \leftskip#3\relax \advance\leftskip\@tempdima\relax
	\rightskip\@pnumwidth plus4em \parfillskip-\@pnumwidth
	#5\leavevmode\hskip-\@tempdima
	\ifcase #1
	\or\or \hskip 1em \or \hskip 2em \else \hskip 3em \fi%
	#6\nobreak\relax
	\hfill\hbox to\@pnumwidth{\@tocpagenum{#7}}\par%
	\nobreak
	\endgroup
	\fi}

\newcommand{\enableopenany}{%
	\@openrightfalse%
}
\makeatother
	
	\setcounter{tocdepth}{1}
	\tableofcontents

\section{\textcolor{bordeau}{Introduction}}
\noindent The notions of formality and coformality originate from rational homotopy theory, which studies spaces up to rational homotopy equivalence.
Two seminal papers on the subject, \cite{Sul77} and \cite{quillen1969rational}, study the rational homotopy type of a given simply connected topological space $X$ through the approach of minimal models, where one associates a minimal model to $X$, that is,
\begin{enumerate}
    \item a differential graded (dg) commutative algebra $S_X$, in Sullivan's approach,
    \item a dg Lie algebra $\lambda_X$, in Quillen's approach,
\end{enumerate}
such that two such spaces have the same rational homotopy type if and only if their minimal models are isomorphic, in either approach. The space $X$ is then called \emph{formal} if there is a quasi-isomorphism of dg commutative algebras $S_X \simeq H^* (X; \mathbb{Q})$, and \emph{coformal} if there is a quasi-isomorphism of dg Lie algebras $\lambda_X \simeq \pi_*(X) \otimes \mathbb{Q}$. These conditions imply that the rational homotopy type of $X$ is determined by its cohomology ring together with the cup product, when $X$ is formal, or by its rational homotopy groups together with the Whitehead bracket, when $X$ is coformal. There is a formulation of Koszul duality such that these two phenomena, formality and coformality, become dual notions \cite{berglund2014koszul}. The existence of suitable geometric structures sometimes constrains a space to be formal and/or coformal; a classical example is the fact that any compact K\"ahler manifold is formal \cite{DGMS75}.

One can phrase both notions, formality and coformality, in terms of the formality of $A_\infty$-algebras, that is, formality of a homotopy operadic algebra structure. Let us recall that more general framework, referring to \cite[Section~2]{Kaledin} for precise statements. We fix a commutative ring $R$ and consider a morphism of operads of the form
\[ \Omega \calC \to \calP, \]
where $\calC$ is a reduced weight-graded cooperad with a canonical twisting morphism to the operad
\[ \calP = \left. \calT(s^{-1}\calC^{(1)}) \middle/ d_{\Omega\calC}(s^{-1}\calC^{(2)}) \right., \]
that is, the operad freely generated by the weight one part of $\calC$, modulo relations coming from the differential of its weight two part. One particular instance of this setting is the resolution $\Omega(\calP^\antish) \to \calP$ of a Koszul operad $\calP$, in which case this morphism is a quasi-isomorphism. However, one does not need to assume this, and it is possible to work with more general $\calC$. Given any such setting $\Omega\calC \to \calP$, a $\Omega\calC$-algebra structure $\varphi$ on a chain complex of $R$-modules $A$ induces a $\calP$-algebra structure $\varphi_*$ on its homology $H = H_*(A)$, which in turn can be seen as a $\Omega\calC$-structure by pulling back.

\medskip
\noindent \textbf{\cref{def:formality}.} \cite[Definition~2.5]{Kaledin}
    A dg $\Omega\calC$-algebra $(A,\varphi)$ is said to be \emph{formal} if there is a zig-zag of quasi-isomorphisms of $\Omega\calC$-algebras
    \[ (A, \varphi) \overset{\sim}{\longleftarrow} \cdot \overset{\sim}{\longrightarrow} \cdots \overset{\sim}{\longleftarrow} \cdot \overset{\sim}{\longrightarrow} (H(A), \varphi_*) \ . \]
    A graded $\calP$-algebra $(H,\varphi_*)$ is said to be \emph{intrinsically formal} if, for any chain complex $A$ with $H_*(A)=H$ and $\Omega\calC$-algebra structure $\varphi$ on $A$ inducing $\varphi_*$ on $H$, the $\Omega\calC$-algebra $(A,\varphi)$ is formal.
\medskip

\noindent To relate this to the setting of rational homotopy theory, one sets $R = \QQ$ and takes the map $\Omega\calC \to \calP$ to be the resolution of the associative operad by the $A_\infty$ operad. Rephrasing the notion of formality, a space $X$ is formal if and only if its dg algebra of singular cochains $C^*(X, \QQ)$, with the cup product, is formal as an $A_\infty$-algebra. As for coformality, a simply connected space $X$ is coformal if and only if the dg algebra $C_*(\Omega X,\QQ)$ of chains in its based (Moore) loop space, with multiplication induced by concatenation of loops, is formal as an $A_\infty$-algebra, by a result of Saleh \cite{Sal17}. In some cases, just the homology-level structures already force formality or coformality: one says that a simply connected space $X$ is intrinsically formal (resp. intrinsically coformal) when $H^*(X,\QQ)$ (resp. $H_*(\Omega X,\QQ)$) is intrinsically formal as a graded associative algebra.

\medskip \medskip

\noindent The cooperads $\calC$ in \cref{def:formality} can be replaced by more general types of algebraic structures, such as codioperads and coproperads. In this paper, we propose to use a certain type of properadic algebra structure, in order to articulate the definition of coformality of a pair $(X,[X])$ of a space endowed with a fundamental class. This type of properadic algebra structure, which we call a $\calY^{(n)}_\infty$-algebra, is defined by a choice of integer $n$, and is given by an extension of $A_\infty$-algebra structures. In the case of our interest, we will study  $\calY^{(n)}_\infty$-algebra structures on chain complexes of the form $A = C_*(\Omega X,R)$, which can be constructed by choosing fundamental class $[X]$ of degree $n$.

The notion of a $\calY^{(n)}_\infty$-algebra structure is a close relative of the notion of a $n$ pre-Calabi--Yau algebra, or $\calV^{(n)}_\infty$-algebra, in the language of Poirier and Tradler \cite{poirier2019koszuality}. These structures are defined by a quadratic dioperad, that is, by operations with multiple inputs and outputs that are required to satisfy quadratic relations involving composition along a single input/output.
This dioperad
\[ \calV^{(n)}_\infty = \Omega  \left( \calV^{(n)\antish} \right) \]
is the cofibrant resolution of a Koszul dioperad $\calV^{(n)}$, encoding associative algebras with a compatible copairing.

The structures relevant to our setting arise from pairs $(X,[X])$ endowed with a degree $n$ fundamental class. We now describe how such a choice of fundamental class gets encoded at the level of properadic algebra structures on the complex of chains on the based loop space $A = C_*(\Omega X,R)$. To be specific, will use the definition of Poincar\'e duality \emph{with local system coefficients}, that is, we require the cap product
\[ \frown [X] \colon H^*(X,\calL) \to H_{n-*}(X,\calL) \]
to be an isomorphism for all local systems $\calL$. We call such a pair $(X,[X])$ a \emph{local Poincar\'e duality pair}. The results of \cite{holstein2024koszul} and \cite{KTV2} imply that there is a non-empty set of $n$-pre-Calabi--Yau/$\calV^{(n)}_\infty$-algebra structures $\varphi$ on $A = C_*(\Omega X,R)$ that are \emph{compatible} with $[X]$, in a sense that we make precise in \cref{def:compatibilityPreCY}.

Making a parallel with the notion of coformality of a space $X$, this suggests that one could try to formulate the notion of coformality of a local Poincar\'e duality pair $(X,[X])$ by substituting $\calV^{(n)}_\infty$ for $A_\infty$ in the setting of \cref{def:formality}. This approach does not yield a very useful notion of coformality, as shown by the following result. \medskip

\noindent \textbf{\cref{thm:notVinftyFormal}.} \textit{When $n \ge 1$, none of the $\calV^{(n)}_\infty$-algebra structures $\varphi$ on $A$ that are compatible with a degree $n$ fundamental class $[X]$ are formal}. \medskip

\noindent In this paper, we will argue that a better notion of coformality of such a pair can be obtained by looking at a slightly different dioperad. Our first observation is that, on algebras that are supported in non-negative homological degree, such as $A = C_*(\Omega X,R)$, the copairing $m_{(2)}^{0,0}$ vanishes automatically when $n \geqslant 1$. Inspired by this observation, we define a coproperad $\calC$ with a quotient map
\[ \calV^{(n)}_\infty \twoheadrightarrow \Omega \calC \]
sending the element corresponding to the copairing $m_{(2)}^{0,0}$ to zero. In other words, a $\Omega \calC$-algebra is an $n$-pre-Calabi--Yau algebra with vanishing copairing. 

We then identify this codioperad $\calC$ as the quadratic dual of a certain dioperad that we call $\calY^{(n)}$, and which we present in \cref{def:YnDioperad}. Denoting $\calY^{(n)}_\infty = \Omega(\calY^{(n)})^\antish$, we have the following characterization of the algebraic structures encoded by this dg dioperad. \medskip

\noindent \textbf{\cref{prop:isomorphismCodioperads}.} 
\textit{There is an isomorphism of codioperads $\calC \cong \calY^{(n)\antish}$, inducing a bijection between $\calY^{(n)}_\infty$-structures and pre-CY-structures with vanishing copairing.
}\medskip

\noindent Though phrased in the setting of dioperads, our constructions can be equivalently made in the language of properads. There is a universal construction $F$ that produces a properad from any dioperad, see \cite{Merkulov_2009}. We prove in \cref{prop:contractible} that the graded properad $F \calY^{(n)}$ has a quadratic dual that is a codioperad, meaning that the following notions agree:
\medskip
\begin{center}
    \emph{dioperadic $\calY^{(n)}_\infty$-algebra} $\; \longleftrightarrow \;$ \emph{properadic $\calY^{(n)}_\infty$-algebra}.
\end{center}
\medskip
\noindent Returning to the case of interest, when $A = C_*(\Omega X,R)$ and $R$ is a $\QQ$-algebra, upon picking a fundamental class $[X]$ of degree $n \ge 1$, the automatic vanishing of the copairing implies that $[X]$ defines a non-empty set of $\calY^{(n)}_\infty$-algebra structures that are compatible with $[X]$, in the sense of \cref{def:compatibilityYinfinity}. Each such structure defines a $\calY^{(n)}$-algebra structure on the homology $H_*(A) = H_*(\Omega X,R)$; we will likewise say that these are the $\calY^{(n)}$-algebra structures on $H_*(A)$ that are compatible with $[X]$. We propose to define the notions of coformality and intrinsic coformality of pairs $(X,[X])$ in the following terms.

\medskip
\noindent \textbf{Definition \ref{def:coformality}.}
    Let $(X,[X])$ be a local Poincar\'e duality pair of degree $n \geqslant 1$.  We say that the pair $(X,[X])$ is
    \begin{itemize}
        \item \emph{coformal} over $R$ when all $\calY^{(n)}_\infty$-algebra structures on $A = C_*(\Omega X,R)$ that are compatible with $[X]$ are formal as $\calY^{(n)}_\infty$-algebras, and
        \item \emph{intrinsically coformal} over $R$ when all $\calY^{(n)}$-algebra structures on $H_*(A) = H_*(\Omega X,R)$ that are compatible with $[X]$ are intrinsically formal as $\calY^{(n)}$-algebras.
    \end{itemize} 
\medskip

\noindent It follows from the definition above, together with \cite{Sal17}, that  coformality of a pair $(X,[X])$ implies coformality in the usual sense of rational homotopy theory, when $X$ is simply connected.
In order to study the formality of these $\calY^{(n)}_\infty$-algebra structures, we will use \emph{properadic Kaledin classes} or obstruction sequences for the deformation theory of properadic algebras developed by the first-named author in \cite{Kaledin,CE24b}. In \cref{sec:intermediateObstruction}, we explain how these obstruction classes get refined in the presence of a second filtration. In \cref{sec:formalityLoopSpaces}, we apply this formalism to study the case when $X$ is a sphere, leading to the following result. \medskip

\noindent \textbf{Theorem \ref{thm:sphereIntrinsicCoformality}.} \textit{For any $n \geqslant 1$, the pair $(S^n,[S^n])$ is \emph{intrinsically coformal}.
}\medskip

\noindent We currently do not know any examples of local Poincar\'e duality pairs $(X,[X])$ that are not coformal in the sense of \cref{def:coformality}, but where $X$ is a coformal space. In other words, examples where $C_*(\Omega X,R)$ is formal as an $A_\infty$-algebra but not as a $\calY^{(n)}_\infty$-algebra.

\medskip \medskip

\noindent We conclude this introduction by mentioning one motivation for defining coformality of a Poincar\'e pair in terms of $\calY^{(n)}_\infty$-algebra structures, coming from string topology. Given any path-connected space $X$, there is a quasi-isomorphism of complexes between chains on its free loop space $C_*(LX,R)$ and Hochschild chains on the algebra $A= C_*(\Omega X,R)$, described in \cite{Goo85}. The Hochschild chain complex of any algebra carries an additional differential $B$ of homological degree $+1$, known as the Connes differential. Under the equivalence above, the resulting mixed complex structure corresponds, up to homotopy, to the one coming from the natural $S^1$-action on $LX$. When $X$ is furthermore endowed with a fundamental class of dimension $n$, we get a quasi-isomorphism 
\[ CH_*(C_*(\Omega X, R)) \cong CH^{n-*}(C_*(\Omega X, R)), \]
that we can use to transfer the Gerstenhaber algebra structure from Hochschild cohomology to Hochschild homology. Combined with the operation $B$, one gets a BV-algebra structure on $HH_*(A)$, see \cite{Tr02}, which agrees with the string topology BV-algebra structure when $X$ is an oriented smooth $n$-manifold \cite{Malm}. 

This BV structure depends on more than the $A_\infty$-algebra structure on $A$, and cannot be recovered just from the Poincar\'e isomorphism at the homology level together with the cup product on cohomology; for example, a naive attempt to reconstruct the BV structure on the free loop space homology $H_*(\Lambda S^2,\mathbb{F}_2)$ starting from a strict Frobenius algebra structure on $H^*(S^2,\mathbb{F}_2)$ fails to recover the correct BV structure \cite{Men09,PoiTra23}. The definition proposed in this paper is an attempt at identifying the type of algebraic formality that would allow one to make computations at the homology level, in the style of the algebraic models proposed in \cite{rivera2023algebraic,rivera2025string}. When a certain oriented smooth manifold $(X,[X])$ is coformal in the sense of \cref{def:coformality}, we expect that the string topology operations on $X$ can be recovered just from the $\calY^{(n)}$-algebra structure on $H_*(\Omega X,R)$. The fact that spheres are formal in this sense should then be seen as the fundamental reason why the examples calculated in \cite{rivera2023algebraic} agree with the geometrically-defined loop product and coproduct for these spaces.

\medskip

\noindent \textbf{Acknowledgments:} The authors thank Alexander Berglund, Vladimir Dotsenko, Sheel Ganatra, Ezra Getzler, Geoffroy Horel, Anton Khoroshkin, Sergei Merkulov, Manuel Rivera, Bruno Vallette and Nathalie Wahl for helpful conversations. We would also like to thank the anonymous referees of this paper for their suggestions and corrections. A.T. would like to thank Uppsala University for the wonderful working environment provided. This work was supported by the Knut and Alice Wallenberg foundation, the project ANR-20-CE40-0016 HighAGT and the \'Ecole Normale Sup\'erieure.

\medskip

\noindent \textbf{Notation and conventions}. 
Throughout this paper, $R$ denotes a commutative $\QQ$-algebra, and we work with $\Z$-graded $R$-modules. We denote by $sA$ the suspension of a graded $R$-module $A$,
and use the abbreviation `dg' for the words `differential graded'. We use the notations and sign conventions of \cite{PHC} for properads; in particular, note that the degree convention for quadratic duals there disagrees with the one used in \cite{poirier2019koszuality}, for instance. Finally, when writing the arity of operations in a properad, we will use the notation (outgoing arity; incoming arity).

\section{\textcolor{bordeau}{Pre-Calabi--Yau algebras with vanishing copairing }}\label{sec:preCY}

\noindent In this section, we begin by recalling the notion of a pre-Calabi--Yau algebra and then focus on the specific case of pre-Calabi--Yau algebras with vanishing copairing.

\subsection{Pre-Calabi--Yau structures}
We start by recalling the definition of a pre-Calabi--Yau algebra structure on a graded $R$-module $A$, see \cite[Section~3]{KTV1} for more details. When $A$ is a degree-wise finite-dimensional graded vector space, this notion appears (up to signs) in \cite{TZ07} under the name of \emph{$V_\infty$-algebra} and in \cite{Sei12} under the name of \emph{boundary algebra}. For all $\ell \geqslant 1$, let us consider the space of $\ell$-higher Hochschild cochains defined by
\[ CH^*_{(\ell)}(A) \coloneqq \prod_{k_1,\dots,k_{\ell} \geqslant 0} \Hom_R \left(sA^{\otimes k_1}\otimes\dots\otimes sA^{\otimes k_\ell},A^{\otimes\ell} \right). \] This graded $R$-module is equipped with a $\Z/\ell$-action given by rotating blocks of inputs and outputs. Let us fix an integer $n$, which we call the \emph{Calabi--Yau dimension}.  We denote by 
\[CH^*_{(\ell)}(A)^{\left(\Z/\ell,n\right)}\]
the isotypic component where the generator acts by the Koszul sign (with inputs and outputs seen as elements of $sA$) times an extra factor $(-1)^{(n-1)(\ell-1)}$. Adding up all those complexes with appropriate shifts gives the \emph{tangent complex}
\[ CH^*_{[n]}(A) \coloneqq \prod_{\ell \geqslant 1} s^{(n-2)(\ell-1)}CH^*_{(\ell)}(A)^{\left(\Z/\ell,n\right)}. \] Every element in this complex decomposes as 
$m=m_{(1)}+ m_{(2)}+ m_{(3)} + \cdots$, where $m_{(\ell)}$ is a cyclically anti/symmetric collection of maps 
\[ m_{(\ell)}^{k_1, \dots, k_{\ell}} \colon sA^{\otimes k_1}\otimes\dots\otimes sA^{\otimes k_\ell} \longrightarrow A^{\otimes\ell} \]
which we  depict by a tree or by a vertex on the plane to emphasize the cyclic group action:
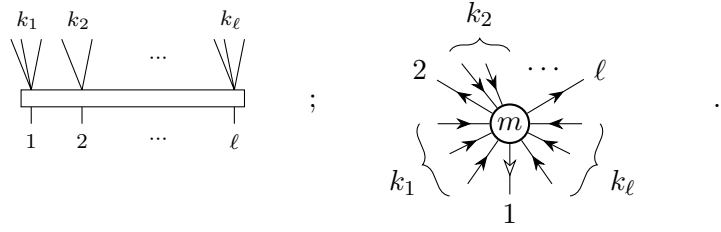
\[
    \begin{tikzpicture}[scale=0.45,baseline=(n.base)]
        \node (n) at (0,0.5) {};
        \coordinate (A) at (0,1);
        \draw[thin]
        (A) ++(0,-1) node[below] {$\scriptstyle{1}$} --
        (A) -- ++(-.6,1.5)
        (A) -- ++(-.3,1.5)
        (A) -- ++(.3,1.5)
        (A)++(-.1,1.5) node[above] {${\scriptstyle{k_1}}$};
        \coordinate (A) at (1.5,1);
        \draw[thin]
        (A) ++(0,-1) node[below] {$\scriptstyle{2}$} --
        (A) -- ++(-.6,1.5)
        (A) -- ++(.3,1.5)
        (A)++(-.1,1.5) node[above] {${\scriptstyle{k_2}}$};
        \draw (3.75,0) node[below] {$\scriptstyle{\cdots}$};
        \draw (3.75,1.5) node[above] {$\scriptstyle{\cdots}$};
        \coordinate (A) at (6,1);
        \draw[thin]
        (A) ++(0,-1) node[below] {$\scriptstyle{\ell}$} --
        (A) -- ++(-.6,1.5)
        (A) -- ++(-.3,1.5)
        (A) -- ++(.3,1.5)
        (A)++(-.1,1.5) node[above] {${\scriptstyle{k_{\ell}}}$};
        \draw[fill=white] (-0.3,0.5) rectangle (6.3,1);
    \end{tikzpicture}	
\qquad  ; \qquad 
\begin{tikzpicture}[scale=0.8,baseline=(n.base)]
        \node [vertex] (m) at (0,0) {$m$};
        \node (b1) at (0,-1.5) {$1$};
        \node (b2) at (-1.5,0.9) {$2$};
        \node at (0.6,0.9) {$\cdots$};
        \node (b3) at (1.5,0.9) {$\ell$};
        \draw [-w-] (m) to (b1);
        \draw [->-] (m) to (b2);
        \draw [->-] (m) to (b3);
        \draw [->-] (-.7,-1) to (m);
        \draw [->-] (-1,-0.5) to (m);
        \draw [->-] (-1.2,0) to (m);
        \draw [->-] (-0.8,1) to (m);
        \draw [->-] (-0.4,1) to (m);
        \draw [->-] (0.7,-1) to (m);
        \draw [->-] (1,-0.5) to (m);
        \draw [->-] (1.2,0) to (m);
        \draw [decorate,decoration={brace,amplitude=5pt}] (-1,-1.2) -- (-1.5,0) node[midway,auto,xshift=-0.1cm]{$k_1$};
        \draw [decorate,decoration={brace,amplitude=5pt}] (-1,1.2) -- (0,1.2) node[midway,auto,yshift=0.2cm]{$k_2$};
        \draw [decorate,decoration={brace,amplitude=5pt,mirror}] (1,-1.2) -- (1.5,0) node[midway,auto,swap,xshift=0.2cm]{$k_\ell$};
    \end{tikzpicture} \qquad \ . 
\]\normalsize  
In the case of a vertex, we denote the first output with a white arrowhead label, and up to sign the $\Z/\ell$-action is given by moving that label around.

\begin{proposition}[{\cite[Proposition~10]{KTV1}}]
    For any $n \in \Z$ and $\ell_1,\ell_2 \geqslant 1$, there is a binary operation called the \emph{necklace product}
    \[ \underset{\mathrm{nec}}{\circ} \colon CH^*_{(\ell_1)}(A)^{\left(\Z/\ell_1,n\right)} \otimes CH^*_{(\ell_2)}(A)^{\left(\Z/\ell_2,n\right)} \to s^{-1}CH^*_{(\ell_1+\ell_2-1)}(A)^{\left(\Z/(\ell_1+\ell_2-1),n\right)}, \]
    which turns the tangent complex into a Lie-admissible algebra  
    \[\mathfrak{g}^{(n)}_{A} \coloneqq \left(s CH^*_{[n]}(A) \ , \underset{\mathrm{nec}}{\circ} \ \right). \]
\end{proposition}

\begin{remark}\label{Hochschild}
    The Hochschild cochain complex of $A$ is a subcomplex of $CH^*_{[n]}(A)$, for any $n$, corresponding to the summand $\ell =1$. The skew symmetrization of the necklace product
    \[ \left[m,m' \right]_{\mathrm{nec}} \coloneqq  m \underset{\mathrm{nec}}{\circ} m' - (-1)^{(|m|-1)(|m'|-1)} m' \underset{\mathrm{nec}}{\circ} m, \]
    defines a Lie bracket extending the Gerstenhaber bracket.  
\end{remark}

\noindent This leads to the following definition due to \cite{poirier2019koszuality} and \cite[Definition~23]{KTV1}.

\begin{definition}
    An \emph{$n$-pre-Calabi--Yau} algebra structure on $A$ is a Maurer--Cartan element \[m \in \mathrm{MC}\left(\mathfrak{g}^{(n)}_{A}\right)\] whose component with zero inputs and one output vanishes, i.e. \[m \underset{\mathrm{nec}}{\circ} m = 0 \quad \mbox{and} \quad  m_{(1)}^0 = 0 \ . \] 
\end{definition}

\begin{remark}
    The component $m_{(1)}$ of an $n$-pre-Calabi--Yau algebra is an $A_\infty$-algebra structure on $A$. Similarly, $m_{(2)}$ is a noncommutative Poisson bivector up to a $[m_{(1)},-]$-exact term since \[ m_{(2)} \underset{\mathrm{nec}}{\circ} m_{(2)} = [m_{(1)},m_{(3)}] \ .\]  
\end{remark}

\noindent It will sometimes be convenient to fix the differential on $A$, $\delta \coloneqq m^{1}_{(1)} \ , $ and look at pre-CY structures extending it. The dg Lie-admissible algebra characterizing those structures is the following one.

\begin{definition}\label{def:dgLieConvolution}
    For any differential $\delta$ on $A$, we define a dg Lie-admissible algebra
    \[ \mathfrak{g}^{(n)}_{(A,\delta)} = \left(s\left(\prod_{r \geqslant 2} \Hom(sA^{\otimes r},A) \times \prod_{\ell \geqslant 2} s^{(n-2)(\ell-1)}CH^*_{(\ell)}(A)^{\left(\Z/\ell,n\right)}\right), \partial, \underset{\mathrm{nec}}{\circ} \right), \]
    where the differential $\partial$ is induced by the internal differential $\delta$ on $A$. An $n$-pre-CY structure extending $\delta$ is a Maurer--Cartan element \[m \in \mathrm{MC}\left(\mathfrak{g}^{(n)}_{(A,\delta)} \right) \ .\] 
\end{definition}

\begin{remark}
    If one drops the requirement that the component with zero inputs and one output vanishes, one gets a definition of a \emph{curved} pre-CY algebra.  This more general notion has been recently shown by Leray and Vallette \cite{LV22} to be equivalent to a curved homotopy version of the notion of double Poisson algebra of \cite{VdB08}. 
\end{remark}

\subsection{The dioperad $\calV_\infty$} Pre-CY algebras were originally named $\calV_\infty$-algebras in \cite{TZ07}; as this name suggests, pre-CY algebras appear as a homotopical version of $\calV$-algebras, which are associative algebras with a compatible copairing. This perspective is explained in detail in \cite{poirier2019koszuality}, using the formalism of Koszul duality for dioperads. 

\begin{definition}[{\cite[Definition~3.1]{PoiTra23}}] For every integer $n$, the $\calV^{(n)}$-dioperad is defined by the following presentation, with two generators $\mu$ and $\nu$ in degree $0$ and $-n$ respectively,
    \[ \calV^{(n)} \coloneqq \frac{\calT\left(\mu =
    \begin{tikzpicture}[baseline=0ex,scale=0.2]
        \draw (0,2) node[above] {$\scriptstyle{1}$}
        -- (2,1);
        \draw (4,2) node[above] {$\scriptstyle{2}$}
        -- (2,1) -- (2,0) node[below] {$\scriptstyle{1}$};
        \draw[fill=white] (2,1) circle (10pt);
    \end{tikzpicture}
    ~ ; ~ \nu = 
    \begin{tikzpicture}[baseline=0ex,scale=0.2]
        \draw (0,0)  node[below] {$\scriptstyle{1}$} -- (0,0.5);
        \draw (2,0) node[below] {$\scriptstyle{2}$} -- (2,0.5);
        \draw[fill=white] (-0.3,0.5) rectangle (2.3,1);
    \end{tikzpicture} ~ = ~
    \begin{tikzpicture}[baseline=0ex,scale=0.2]
        \draw (0,0)  node[below] {$\scriptstyle{2}$} -- (0,0.5);
        \draw (2,0) node[below] {$\scriptstyle{1}$} -- (2,0.5);
        \draw[fill=white] (-0.3,0.5) rectangle (2.3,1);
    \end{tikzpicture}
    \right)}
    {\left(\begin{tikzpicture}[baseline=0.5ex,scale=0.2]
        \draw (0,4) node[above] {$\scriptstyle{1}$} --(4,0)
        -- (4,-1) node[below] {$\scriptstyle{1}$};
        \draw (4,4) node[above] {$\scriptstyle{2}$} -- (2,2);
        \draw (8,4) node[above] {$\scriptstyle{3}$}-- (4,0);
        \draw[fill=white] (2,2) circle (10pt);
        \draw[fill=white] (4,0) circle (10pt);
    \end{tikzpicture}
    -
    \begin{tikzpicture}[baseline=0.5ex,scale=0.2]
        \draw (0,4) node[above] {$\scriptstyle{1}$}
        -- (4,0) -- (4,-1) node[below] {$\scriptstyle{1}$};
        \draw (4,4) node[above] {$\scriptstyle{2}$} -- (6,2);
        \draw (8,4) node[above] {$\scriptstyle{3}$} -- (4,0);
        \draw[fill=white] (6,2) circle (10pt);
        \draw[fill=white] (4,0) circle (10pt);
    \end{tikzpicture}
    ~;~
    \begin{tikzpicture}[scale=0.2,baseline=1ex]
        \draw (0,4) node[above] {$\scriptstyle{1}$} -- (0,2) -- (1,0.5) -- (1,-0.5) node[below] {$\scriptstyle{1}$};
        \draw (4,-0.5) node[below] {$\scriptstyle{2}$} -- (4,2) ;
        \draw (2,2) -- (1,0.5);
        \draw[fill=white] (1,0.5) circle (10pt);
        \draw[fill=white] (1.7,2) rectangle (4.3,2.5);
    \end{tikzpicture}
    ~ - ~
    \begin{tikzpicture}[scale=0.2,baseline=1ex]
        \draw (0,4) node[above] {$\scriptstyle{1}$} -- (0,2) -- (-1,0.5) -- (-1,-0.5) node[below] {$\scriptstyle{2}$};
        \draw (-4,-0.5) node[below] {$\scriptstyle{1}$} -- (-4,2) ;
        \draw (-2,2) -- (-1,0.5);
        \draw[fill=white] (-1,0.5) circle (10pt);
        \draw[fill=white] (-1.7,2) rectangle (-4.3,2.5);
    \end{tikzpicture} \right)} \ .\]
\end{definition}

\noindent A $\calV^{(n)}$-algebra structure is thus determined by a binary product $\mu$ and a symmetric copairing $\nu$ of degree $-n$, satisfying
\[ \mu(x,\nu') \otimes \nu'' = (-1)^{n|x|} \nu' \otimes \mu(\nu'',x) \]
for all $x \in A$, where we used the Sweedler notation ``$\nu = \nu' \otimes \nu''$''. The results of Poirier and Tradler can be rephrased in the following way.

\begin{theorem}[{\cite[Theorem~1.2 and Proposition~3.4]{poirier2019koszuality}}] \leavevmode

\begin{enumerate}
    \item The dioperad $\calV^{(n)}$ is Koszul and we denote $\calV^{(n)}_\infty \coloneqq \Omega \calV^{(n) \antish}$. 
    \item  For any chain complex $(A,\delta)$, there is an isomorphism of dg Lie-admissible algebras
    \[ \mathfrak{g}^{(n)}_{\left(A,\delta \right)} \cong \left(s\Hom_{\S}\left(\calV^{(n)\antish}, \End_A \right), \partial, \star \right) \]
    between the dg Lie-admissible algebra of \cref{def:dgLieConvolution} and the convolution algebra governing $\calV^{(n)}_\infty$-algebra structures on $(A,\delta)$.
\end{enumerate}
\end{theorem}

\subsection{Pre-CY algebras with vanishing copairing}

In this paper, we will restrict our attention to the following specific type of pre-CY structure.

\begin{definition}
    An $n$-pre-CY algebra structure \emph{with vanishing copairing} on $A$ is an $n$-pre-CY structure on $A$ such that its component with zero inputs and two outputs vanishes, i.e. 
    \[m^{0,0}_{(2)}=0 \ .\] Such a structure is given by operations
    \[\begin{array}{lllllll}
        m^{i\geqslant 1}_{(1)} & \colon &  sA^{\otimes i} \to s^2A & \quad \quad \quad& && \medskip \\
        m^{1,0}_{(2)} &\colon &  sA \to s^n A^{\otimes 2}& \quad \quad \quad& m^{1,1}_{(2)}&\colon & sA \otimes sA \to s^n A^{\otimes 2} \quad \dots \medskip \\
        m^{0,0,0}_{(3)} &\colon & R \to s^{2n-2} A^{\otimes 3} & & m^{1,0,0}_{(3)} &\colon & sA \to s^{2n-2}A^{\otimes 3}\quad \dots
    \end{array}\] 
    The other operations of the same arities are determined by symmetry. For instance, we have 
    \[m^{0,1}_{(2)} = (-1)^{n-1} \tau \circ m^{1,0}_{(2)} \ ,\]
    where $\tau$ exchanges factors, with the Koszul sign given by seeing them as elements of $sA$. 
\end{definition}

\begin{example}\label{prop:connectiveVanishing}
    If $A$ is connective, that is supported in non-negative homological degree, and $n \geqslant 1$, then any $n$-pre-CY structure on $A$ has vanishing copairing.
\end{example}

\noindent In the $\calV_\infty$ perspective, we have the following equivalent characterization. Let us denote by $\calC$ the graded $\S$-bimodule given by the quotient $\calV^{(n)\antish} \twoheadrightarrow \C$
that kills the component of arity $(2;0)$.
\begin{lemma}
    The codioperad structure of $\calV^{(n)\antish}$ induces a codioperad structure on $\C$.
\end{lemma}
\begin{proof}
    This follows from the fact that the infinitesimal decomposition morphism of $\calV^{(n)\antish}$ is compatible with the quotient. The only term that could cause a problem is the decomposition of an arity $(2;0)$ operation into an arity $(2;1)$ operation and an arity $(1;0)$ operation; this does not occur since the arity $(1;0)$ component of $\calV^{(n) \antish}$ vanishes.
\end{proof}

\begin{proposition}
    There is a natural bijection between the set of $n$-pre-CY algebra structures with vanishing copairing on a chain complex $(A,\delta)$ and the set of $\Omega \C$-algebra structures on the same chain complex. In other words, $n$-pre-CY structures with vanishing copairing are controlled by the dg Lie algebra
    \[\left(s\Hom_{\S}(\calC, \End_A), \partial, [,] \right).\]
\end{proposition}
\begin{proof}
    The  surjection $\calV^{(n)\antish} \twoheadrightarrow C$ induces a surjection of dg dioperads \[ \Omega \calV^{(n)^\antish} \twoheadrightarrow \Omega \calC \ .\] This exhibits the set of $\Omega \calC$-algebra structures as the subset of $\calV^{(n)}_\infty$-algebra structures with vanishing copairing.
\end{proof}

\subsection{The dioperad $\calY^{(n)}$}
We now characterize the dioperad $\calC$ as the quadratic dual of a certain dioperad. For that, we look at the lowest-arity non-trivial operations of a pre-CY algebra with vanishing copairing.
We denote by
\[ V = R \cdot \big(\id + (231) + (312)\big) \oplus R \cdot \big((132)+ (321) + (213) \big) \subset R[\S_3] \]
the rank two subspace of the regular representation on which cyclic permutations act trivially.
\begin{definition}\label{def:YnDioperad}
    For each integer $n$, we define \[\calY^{(n)} = \calT(E)/ (R) \] to be the quadratic dioperad generated by
    \[\begin{array}{ccccccc}        
        E   = &E(1;2)   & \oplus    & E(2;1) & \oplus & E(3;0)\\
            &R[\S_2]\cdot \mu  & \oplus    & R[\S_2]\cdot \alpha & \oplus & V \cdot \beta
    \end{array}\]
    where $|\mu| = 0$, $|\alpha| = -n +1$, and $|\beta| = -2n + 2$. We depict these generators as follows \tiny
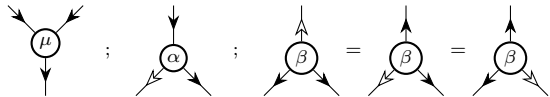
    \[
    \tikzfig{
        \node[vertex] (mu) at (0,0) {$\mu$};
        \draw [->-] (-0.5,0.5) to (mu);
        \draw [->-] (0.5,0.5)to (mu);
        \draw [->-] (mu) to (0,-0.7);
    } \quad ; \quad \tikzfig{
        \node[vertex] (a) at (0,0) {$\alpha$};
        \draw [-w-] (a) to (-0.5,-0.5);
        \draw [->-] (a) to (0.5,-0.5);
        \draw [->-] (0,0.7) to (a);
    } \quad ; \quad \tikzfig{
        \node[vertex] (b) at (0,0) {$\beta$};
        \draw [->-] (b) to (-0.5,-0.5);
        \draw [->-] (b) to (0.5,-0.5);
        \draw [-w-] (b) to (0,0.7);
    } = \tikzfig{
        \node[vertex] (b) at (0,0) {$\beta$};
        \draw [-w-] (b) to (-0.5,-0.5);
        \draw [->-] (b) to (0.5,-0.5);
        \draw [->-] (b) to (0,0.7);
    } = \tikzfig{
        \node[vertex] (b) at (0,0) {$\beta$};
        \draw [->-] (b) to (-0.5,-0.5);
        \draw [-w-] (b) to (0.5,-0.5);
        \draw [->-] (b) to (0,0.7);
    }\] \normalsize
    we can express the submodule of relations $R$ as generated by the elements 
    \medskip
    
    \tiny
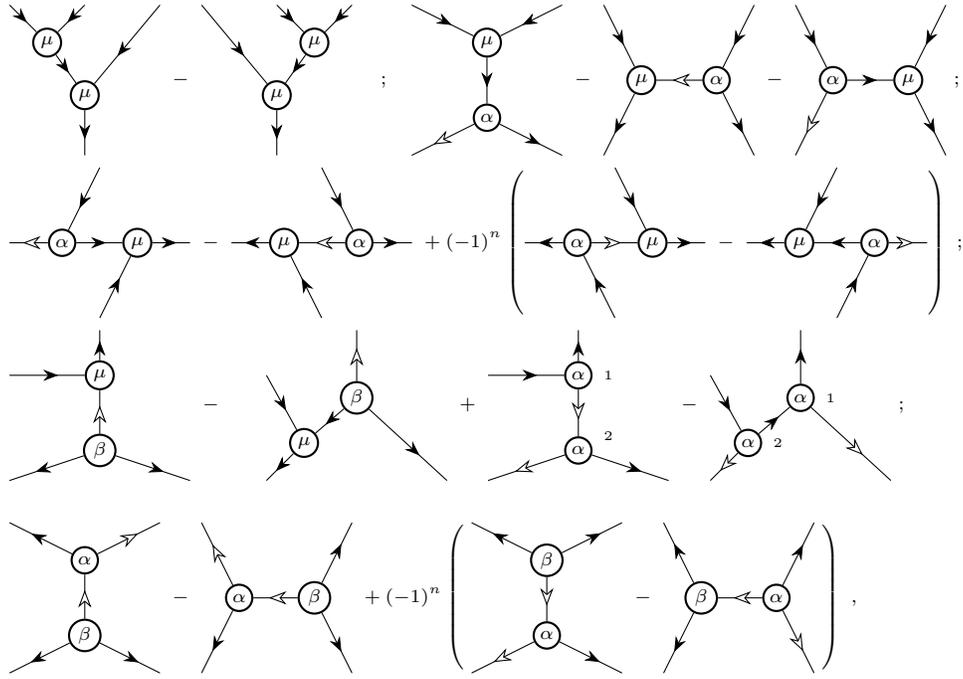
    \begin{align*}
        &\tikzfig{
            \node [vertex] (bot) at (0,-0.2) {$\mu$};
            \node [vertex] (top) at (-0.5,0.5) {$\mu$};
            \draw [->-] (-1,1) to (top);
            \draw [->-] (0,1) to (top);
            \draw [->-] (top) to (bot);
            \draw [->-] (1,1) to (bot);
            \draw [->-] (bot) to (0,-1);
        } ~ - ~ \tikzfig{
            \node [vertex] (bot) at (0,-0.2) {$\mu$};
            \node [vertex] (top) at (0.5,0.5) {$\mu$};
            \draw [->-] (-1,1) to (bot);
            \draw [->-] (0,1) to (top);
            \draw [->-] (top) to (bot);
            \draw [->-] (1,1) to (top);
            \draw [->-] (bot) to (0,-1);
        } ~ \quad ; \quad \tikzfig{
            \node [vertex] (top) at (0,0.5) {$\mu$};
            \node [vertex] (bot) at (0,-0.5) {$\alpha$};
            \draw [->-] (-1,1) to (top);
            \draw [->-] (top) to (bot);
            \draw [->-] (1,1) to (top);
            \draw [-w-] (bot) to (-1,-1);
            \draw [->-] (bot) to (1,-1);
        } ~ - ~ \tikzfig{
            \node [vertex] (top) at (0.5,0) {$\alpha$};
            \node [vertex] (bot) at (-0.5,0) {$\mu$};
            \draw [->-] (-1,1) to (bot);
            \draw [-w-] (top) to (bot);
            \draw [->-] (1,1) to (top);
            \draw [->-] (bot) to (-1,-1);
            \draw [->-] (top) to (1,-1);
        } ~ - ~ \tikzfig{
            \node [vertex] (top) at (-0.5,0) {$\alpha$};
            \node [vertex] (bot) at (0.5,0) {$\mu$};
            \draw [->-] (-1,1) to (top);
            \draw [->-] (top) to (bot);
            \draw [->-] (1,1) to (bot);
            \draw [-w-] (top) to (-1,-1);
            \draw [->-] (bot) to (1,-1);
        } ~;\\
        &\tikzfig{
            \node [vertex] (left) at (-0.5,0) {$\alpha$};
            \node [vertex] (right) at (0.5,0) {$\mu$};
            \draw [->-] (0,1) to (left);
            \draw [->-] (0,-1) to (right);
            \draw [-w-] (left) to (-1.2,0);
            \draw [->-] (left) to (right);
            \draw [->-] (right) to (1.2,0);
        } ~ - ~ \tikzfig{
            \node [vertex] (left) at (-0.5,0) {$\mu$};
            \node [vertex] (right) at (0.5,0) {$\alpha$};
            \draw [->-] (0,1) to (right);
            \draw [->-] (0,-1) to (left);
            \draw [->-] (left) to (-1.2,0);
            \draw [-w-] (right) to (left);
            \draw [->-] (right) to (1.2,0);
        } ~ + (-1)^n \left(\tikzfig{
            \node [vertex] (left) at (-0.5,0) {$\alpha$};
            \node [vertex] (right) at (0.5,0) {$\mu$};
            \draw [->-] (0,1) to (right);
            \draw [->-] (0,-1) to (left);
            \draw [->-] (left) to (-1.2,0);
            \draw [-w-] (left) to (right);
            \draw [->-] (right) to (1.2,0);
        } ~ - ~ \tikzfig{
            \node [vertex] (left) at (-0.5,0) {$\mu$};
            \node [vertex] (right) at (0.5,0) {$\alpha$};
            \draw [->-] (0,1) to (left);
            \draw [->-] (0,-1) to (right);
            \draw [->-] (left) to (-1.2,0);
            \draw [->-] (right) to (left);
            \draw [-w-] (right) to (1.2,0);
        }\right) ~; \\
        &\tikzfig{
            \node [vertex] (top) at (0,0.4) {$\mu$};
            \node [vertex] (bot) at (0,-0.6) {$\beta$};
            \draw [-w-] (bot) to (top);
            \draw [->-] (bot) to (-1.2,-1);
            \draw [->-] (bot) to (1.2,-1);
            \draw [->-] (-1.2,0.4) to (top);
            \draw [->-] (top) to (0,1);
        } ~ - \qquad \tikzfig{
            \node [vertex] (top) at (0,0.1) {$\beta$};
            \node [vertex] (bot) at (-0.7,-0.5) {$\mu$};
            \draw [-w-] (top) to (0,1);
            \draw [->-] (bot) to (-1.2,-1);
            \draw [->-] (top) to (1.2,-1);
            \draw [->-] (-1.2,0.4) to (bot);
            \draw [->-] (top) to (bot);
        } ~ + ~ \tikzfig{
            \node [vertex] (top) at (0,0.4) {$\alpha$};
            \node [vertex] (bot) at (0,-0.6) {$\alpha$};
            \node at (0.4,0.4) {$\scriptstyle{1}$};
            \node at (0.4,-0.4) {$\scriptstyle{2}$};
            \draw [-w-] (top) to (bot);
            \draw [-w-] (bot) to (-1.2,-1);
            \draw [->-] (bot) to (1.2,-1);
            \draw [->-] (-1.2,0.4) to (top);
            \draw [->-] (top) to (0,1);
        } ~ - ~ \tikzfig{
            \node [vertex] (top) at (0,0.1) {$\alpha$};
            \node [vertex] (bot) at (-0.7,-0.5) {$\alpha$};
            \node at (0.4,0.1) {$\scriptstyle{1}$};
            \node at (-0.3,-0.5) {$\scriptstyle{2}$};
            \draw [->-] (top) to (0,1);
            \draw [-w-] (bot) to (-1.2,-1);
            \draw [-w-] (top) to (1.2,-1);
            \draw [->-] (-1.2,0.4) to (bot);
            \draw [->-] (bot) to (top);
        } ~; \\ \\
        &\tikzfig{
            \node [vertex] (top) at (0,0.5) {$\alpha$};
            \node [vertex] (bot) at (0,-0.5) {$\beta$};
            \draw [-w-] (bot) to (top);
            \draw [->-] (top) to (-1,1);
            \draw [-w-] (top) to (1,1);
            \draw [->-] (bot) to (-1,-1);
            \draw [->-] (bot) to (1,-1);
        } ~ - ~ \tikzfig{
            \node [vertex] (left) at (-0.5,0) {$\alpha$};
            \node [vertex] (right) at (0.5,0) {$\beta$};
            \draw [-w-] (right) to (left);
            \draw [-w-] (left) to (-1,1);
            \draw [->-] (left) to (-1,-1);
            \draw [->-] (right) to (1,-1);
            \draw [->-] (right) to (1,1);
        } ~ + (-1)^n \left( ~ \tikzfig{
            \node [vertex] (top) at (0,0.5) {$\beta$};
            \node [vertex] (bot) at (0,-0.5) {$\alpha$};
            \draw [-w-] (top) to (bot);
            \draw [->-] (top) to (-1,1);
            \draw [->-] (top) to (1,1);
            \draw [-w-] (bot) to (-1,-1);
            \draw [->-] (bot) to (1,-1);
        } ~ - ~ \tikzfig{
            \node [vertex] (left) at (-0.5,0) {$\beta$};
            \node [vertex] (right) at (0.5,0) {$\alpha$};
            \draw [-w-] (right) to (left);
            \draw [->-] (left) to (-1,1);
            \draw [->-] (left) to (-1,-1);
            \draw [-w-] (right) to (1,-1);
            \draw [->-] (right) to (1,1);
        } ~ \right) ~, 
    \end{align*}\normalsize
\medskip
    
\noindent where we use the labels in the last two terms of the third line to indicate in which order we input the $\alpha$ operations before evaluating.
\end{definition}

\noindent  In more conventional notation, where one insists in having inputs on top and outputs on the bottom, we can write
\[ \calY^{(n)} = \left.\calT\left(
\begin{tikzpicture}[baseline=0ex,scale=0.2]
    \draw (0,2) node[above] {$\scriptstyle{1}$}
    -- (2,1);
   \draw (4,2) node[above] {$\scriptstyle{2}$}
   -- (2,1) -- (2,0) node[below] {$\scriptstyle{1}$};
   \draw[fill=white] (2,1) circle (10pt);
\end{tikzpicture}
~ ; ~
\begin{tikzpicture}[baseline=0ex,scale=0.2]
   \draw (0,0)  node[below] {$\scriptstyle{1}$} -- (0,0.5);
   \draw (2,0) node[below] {$\scriptstyle{2}$} -- (2,0.5);
   \draw[fill=white] (-0.3,0.5) rectangle (2.3,1);
  \draw (1,2) node[above] {$\scriptstyle{1}$} -- (1,1);
\end{tikzpicture}
~ ; ~
\begin{tikzpicture}[baseline=0ex,scale=0.2]
   \draw (0,0)  node[below] {$\scriptstyle{1}$} -- (0,0.5);
   \draw (2,0) node[below] {$\scriptstyle{2}$} --  (2,0.5);
   \draw (4,0) node[below] {$\scriptstyle{3}$} -- (4,0.5);
   \draw[fill=gray] (-0.3,0.5) rectangle (4.3,1);
\end{tikzpicture}
~ = ~
\begin{tikzpicture}[baseline=0ex,scale=0.2]
   \draw (0,0)  node[below] {$\scriptstyle{2}$} -- (0,0.5);
   \draw (2,0) node[below] {$\scriptstyle{3}$} --  (2,0.5);
   \draw (4,0) node[below] {$\scriptstyle{1}$} -- (4,0.5);
   \draw[fill=gray] (-0.3,0.5) rectangle (4.3,1);
\end{tikzpicture}
~ = ~
\begin{tikzpicture}[baseline=0ex,scale=0.2]
   \draw (0,0)  node[below] {$\scriptstyle{3}$} -- (0,0.5);
   \draw (2,0) node[below] {$\scriptstyle{1}$} --  (2,0.5);
   \draw (4,0) node[below] {$\scriptstyle{2}$} -- (4,0.5);
   \draw[fill=gray] (-0.3,0.5) rectangle (4.3,1);
\end{tikzpicture}
\right)
\right/ (R) \]
with $R$ generated by the elements
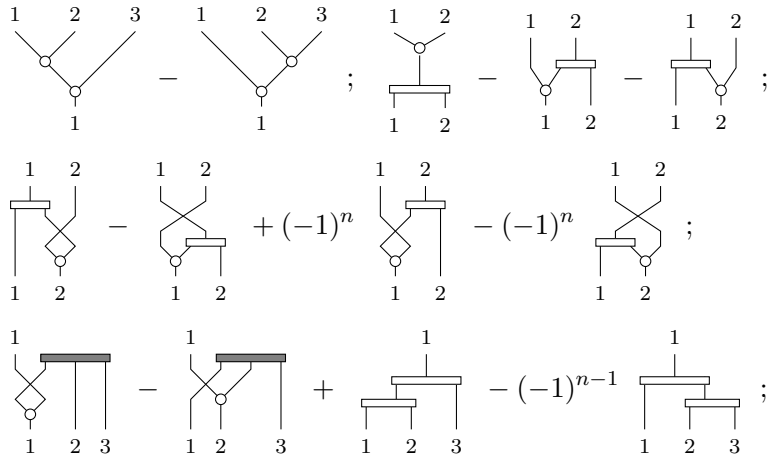
\begin{align*} 
&\begin{tikzpicture}[baseline=0.5ex,scale=0.2]
  \draw (0,4) node[above] {$\scriptstyle{1}$} --(4,0)
   -- (4,-1) node[below] {$\scriptstyle{1}$};
   \draw (4,4) node[above] {$\scriptstyle{2}$} -- (2,2);
   \draw (8,4) node[above] {$\scriptstyle{3}$}-- (4,0);
   \draw[fill=white] (2,2) circle (10pt);
   \draw[fill=white] (4,0) circle (10pt);
\end{tikzpicture}
-
\begin{tikzpicture}[baseline=0.5ex,scale=0.2]
   \draw (0,4) node[above] {$\scriptstyle{1}$}
   -- (4,0) -- (4,-1) node[below] {$\scriptstyle{1}$};
   \draw (4,4) node[above] {$\scriptstyle{2}$} -- (6,2);
   \draw (8,4) node[above] {$\scriptstyle{3}$} -- (4,0);
   \draw[fill=white] (6,2) circle (10pt);
   \draw[fill=white] (4,0) circle (10pt);
\end{tikzpicture}
~;~~
\begin{tikzpicture}[scale=0.2,baseline=0ex]
   \draw (0.3,3.5) node[above] {$\scriptstyle{1}$} --(2,2.5);
   \draw (3.7,3.5) node[above] {$\scriptstyle{2}$} -- (2,2.5);
   \draw (2,2) -- (2,0);
   \draw (0.3,0) -- (0.3,-1.5) node[below] {$\scriptstyle{1}$};
   \draw (3.7,0) -- (3.7,-1.5) node[below] {$\scriptstyle{2}$};
   \draw[fill=white] (2,2.5) circle (10pt);
   \draw[fill=white] (0,-0.5) rectangle (4,0);
\end{tikzpicture}
~ - ~
\begin{tikzpicture}[scale=0.2,baseline=1ex]
   \draw (0,4) node[above] {$\scriptstyle{1}$} -- (0,2) -- (1,0.5) -- (1,-0.5) node[below] {$\scriptstyle{1}$};
   \draw (4,-0.5) node[below] {$\scriptstyle{2}$} -- (4,2) ;
   \draw (3,4) node[above] {$\scriptstyle{2}$} -- (3,2.5); 
   \draw (2,2) -- (1,0.5);
   \draw[fill=white] (1,0.5) circle (10pt);
   \draw[fill=white] (1.7,2) rectangle (4.3,2.5);
\end{tikzpicture}
~ - ~
\begin{tikzpicture}[scale=0.2,baseline=1ex]
   \draw (0,4) node[above] {$\scriptstyle{2}$} -- (0,2) -- (-1,0.5) -- (-1,-0.5) node[below] {$\scriptstyle{2}$};
   \draw (-4,-0.5) node[below] {$\scriptstyle{1}$} -- (-4,2) ;
   \draw (-3,4) node[above] {$\scriptstyle{1}$} -- (-3,2.5); 
   \draw (-2,2) -- (-1,0.5);
   \draw[fill=white] (-1,0.5) circle (10pt);
   \draw[fill=white] (-1.7,2) rectangle (-4.3,2.5);
\end{tikzpicture} ~;\\
&\begin{tikzpicture}[scale=0.2,baseline=1ex]
   \draw (0,4) node[above] {$\scriptstyle{2}$} -- (0,2) -- (-2,0) -- (-1,-1) -- (-1,-2) node[below] {$\scriptstyle{2}$};
   \draw (-4,-2) node[below] {$\scriptstyle{1}$} -- (-4,2.5) ;
    \draw (-3,4) node[above] {$\scriptstyle{1}$} -- (-3,3); 
   \draw (-2,3) -- (-2,2) -- (0,0) -- (-1,-1);
   \draw[fill=white] (-1,-1) circle (10pt);
   \draw[fill=white] (-1.7,2.5) rectangle (-4.3,3);
\end{tikzpicture} 
~ - ~
\begin{tikzpicture}[scale=0.2,baseline=1ex]
   \draw (0,4) node[above] {$\scriptstyle{1}$} -- (0,3) -- (3,1) -- (3,0.5);
   \draw (4,-2) node[below] {$\scriptstyle{2}$} -- (4,0) ;
   \draw (3,4) node[above] {$\scriptstyle{2}$} -- (3,3) -- (0,1) -- (0,0) -- (1,-1) -- (1,-2) node[below] {$\scriptstyle{1}$}; 
   \draw (2,0) -- (1,-1);
   \draw[fill=white] (1,-1) circle (10pt);
   \draw[fill=white] (1.7,0) rectangle (4.3,0.5);
\end{tikzpicture} ~ + (-1)^n ~
\begin{tikzpicture}[scale=0.2,baseline=1ex]
   \draw (0,4) node[above] {$\scriptstyle{1}$} -- (0,2) -- (2,0) -- (1,-1) -- (1,-2) node[below] {$\scriptstyle{1}$};
   \draw (4,-2) node[below] {$\scriptstyle{2}$} -- (4,2.5) ;
   \draw (3,4) node[above] {$\scriptstyle{2}$} -- (3,3); 
   \draw (2,3) -- (2,2) -- (0,0) -- (1,-1);
   \draw[fill=white] (1,-1) circle (10pt);
   \draw[fill=white] (1.7,2.5) rectangle (4.3,3);
\end{tikzpicture} ~ - (-1)^n~
\begin{tikzpicture}[scale=0.2,baseline=1ex]
   \draw (0,4) node[above] {$\scriptstyle{2}$} -- (0,3) -- (-3,1) -- (-3,0.5);
   \draw (-4,-2) node[below] {$\scriptstyle{1}$} -- (-4,0) ;
   \draw (-3,4) node[above] {$\scriptstyle{1}$} -- (-3,3) -- (0,1) -- (0,0) -- (-1,-1) -- (-1,-2) node[below] {$\scriptstyle{2}$}; 
   \draw (-2,0) -- (-1,-1);
   \draw[fill=white] (-1,-1) circle (10pt);
   \draw[fill=white] (-1.7,0) rectangle (-4.3,0.5);
\end{tikzpicture}~;\\
&\begin{tikzpicture}[baseline=-1ex,scale=0.2]
   \draw (-1,-3) node[below] {$\scriptstyle{1}$} -- (-1,-2);
   \draw (2,-3) node[below] {$\scriptstyle{2}$} --  (2,1.5);
   \draw (4,-3) node[below] {$\scriptstyle{3}$} -- (4,1.5);
   \draw (-2,2) node[above] {$\scriptstyle{1}$} -- (-2,1) -- (0,-1) -- (-1,-2);
   \draw (0,1.5) -- (0,1) -- (-2,-1) -- (-1,-2);
   \draw[fill=gray] (-0.3,1.5) rectangle (4.3,2);
   \draw[fill=white] (-1,-2) circle (10pt);
\end{tikzpicture} ~ - ~ 
\begin{tikzpicture}[baseline=-1ex,scale=0.2]
   \draw (2,1.5) -- (2,1) -- (0,-1) -- (0,-3) node[below] {$\scriptstyle{2}$};
   \draw (4,-3) node[below] {$\scriptstyle{3}$} -- (4,1.5);
   \draw (-2,2) node[above] {$\scriptstyle{1}$} -- (-2,1) -- (0,-1);
   \draw (0,1.5) -- (0,1) -- (-2,-1) -- (-2,-3) node[below] {$\scriptstyle{1}$};
   \draw[fill=gray] (-0.3,1.5) rectangle (4.3,2);
   \draw[fill=white] (0,-1) circle (10pt);
\end{tikzpicture}
~ + ~
\begin{tikzpicture}[baseline=-1ex,scale=0.2]
   \draw (2,2) node[above] {$\scriptstyle{1}$} -- (2,0.5);
   \draw (0,0.5) -- (0,-1);
   \draw (4,0) -- (4,-3) node[below] {$\scriptstyle{3}$}; 
   \draw (1,-1.5) to (1,-3) node[below] {$\scriptstyle{2}$};
   \draw (-2,-1.5) to (-2,-3) node[below] {$\scriptstyle{1}$};
   \draw[fill=white] (-0.3,0) rectangle (4.3,0.5);
   \draw[fill=white] (-2.3,-1.5) rectangle (1.3,-1);
\end{tikzpicture}
~ -(-1)^{n-1} ~
\begin{tikzpicture}[baseline=-1ex,scale=0.2]
   \draw (-2,2) node[above] {$\scriptstyle{1}$} -- (-2,0.5);
   \draw (0,0.5) -- (0,-1);
   \draw (-4,0) -- (-4,-3) node[below] {$\scriptstyle{1}$}; 
   \draw (-1,-1.5) to (-1,-3) node[below] {$\scriptstyle{2}$};
   \draw (2,-1.5) to (2,-3) node[below] {$\scriptstyle{3}$};
   \draw[fill=white] (0.3,0) rectangle (-4.3,0.5);
   \draw[fill=white] (2.3,-1.5) rectangle (-1.3,-1);
\end{tikzpicture}~;\\
&\begin{tikzpicture}[baseline=-1ex,scale=0.2]
   \draw (-2,-3) node[below] {$\scriptstyle{1}$} -- (-2,-1.5);
   \draw (0,-3) node[below] {$\scriptstyle{2}$} -- (0,-1.5);
   \draw (2,-3) node[below] {$\scriptstyle{3}$} --  (2,1.5);
   \draw (4,-3) node[below] {$\scriptstyle{4}$} -- (4,1.5);
   \draw (0,1.5) -- (0,1) -- (-1,0) -- (-1,-1);
   \draw[fill=gray] (-0.3,1.5) rectangle (4.3,2);
   \draw[fill=white] (-2.3,-1.5) rectangle (0.3,-1);
\end{tikzpicture} ~ - ~ 
\begin{tikzpicture}[baseline=-1ex,scale=0.2]
   \draw (-2,-3) node[below] {$\scriptstyle{2}$} -- (-2,-1.5);
    \draw (0,-3) node[below] {$\scriptstyle{3}$} -- (0,-1.5);
   \draw (2,-3) node[below] {$\scriptstyle{4}$} --  (2,1.5);
   \draw (4,-3) node[below] {$\scriptstyle{1}$} -- (4,1.5);
   \draw (0,1.5) -- (0,1) -- (-1,0) -- (-1,-1);
   \draw[fill=gray] (-0.3,1.5) rectangle (4.3,2);
   \draw[fill=white] (-2.3,-1.5) rectangle (0.3,-1);
\end{tikzpicture} ~ +(-1)^n ~
\begin{tikzpicture}[baseline=-1ex,scale=0.2]
   \draw (-2,-3) node[below] {$\scriptstyle{3}$} -- (-2,-1.5);
   \draw (0,-3) node[below] {$\scriptstyle{4}$} -- (0,-1.5);
   \draw (2,-3) node[below] {$\scriptstyle{1}$} --  (2,1.5);
   \draw (4,-3) node[below] {$\scriptstyle{2}$} -- (4,1.5);
   \draw (0,1.5) -- (0,1) -- (-1,0) -- (-1,-1);
   \draw[fill=gray] (-0.3,1.5) rectangle (4.3,2);
   \draw[fill=white] (-2.3,-1.5) rectangle (0.3,-1);
\end{tikzpicture} ~  -(-1)^n ~
\begin{tikzpicture}[baseline=-1ex,scale=0.2]
   \draw (-2,-3) node[below] {$\scriptstyle{4}$} -- (-2,-1.5);
   \draw (0,-3) node[below] {$\scriptstyle{1}$} -- (0,-1.5);
   \draw (2,-3) node[below] {$\scriptstyle{2}$} --  (2,1.5);
   \draw (4,-3) node[below] {$\scriptstyle{3}$} -- (4,1.5);
   \draw (0,1.5) -- (0,1) -- (-1,0) -- (-1,-1);
   \draw[fill=gray] (-0.3,1.5) rectangle (4.3,2);
   \draw[fill=white] (-2.3,-1.5) rectangle (0.3,-1);
\end{tikzpicture}.
\end{align*}

\begin{proposition}\label{prop:YimpliesPreCY}
    If $(\mu,\alpha,\beta)$ is a $\calY^{(n)}$-algebra structure on $A$ with vanishing differential, then
    \[ m^2_{(1)} = \pm \mu, \quad m^{1,0}_{(2)} = (-1)^{n-1} \tau \circ m^{0,1}_{(2)} = \pm \alpha, \quad m^{0,0,0}_{(3)} = \pm \beta, \]
    for an appropriate choice of signs, with all the other structure maps set to zero, defines a pre-CY structure with vanishing copairing on $A$.
\end{proposition}
\begin{proof}
    Follows from the fact that if the differential of $A$ vanishes the pre-CY structure equation $[m,m]=0$ decomposes into a closed set of equations for $m^2_{(1)}, m^{1,0}_{(2)}$ and $m^{0,0,0}_{(3)}$, in the sense that they do not involve any of the higher operations. By changing sign conventions between e.g. \cite{KTV1} and the conventions above, we check that these equations agree with the $\calY^{(n)}$-algebra structure equations.
\end{proof}

\noindent Let us consider the quadratic dual dioperad of $\calY^{(n)}$. This can be given explicitly by using the description of \cite[Corollary~7.12]{Vallette_2007}, keeping only the relations of genus zero. One obtains the following presentation \[\calY^{(n)!} = \calT \left(E^\vee \right)/ (R^\perp_0 ) \] where the generators are
    \[\begin{array}{lllllll}        
        E^\vee   = &E^\vee(1;2)   & \oplus    & E^\vee(2;1) & \oplus & E^\vee(3;0)\\
            &R[\S_2]\cdot \nu  & \oplus    & R[\S_2]\cdot \omega & \oplus & V\otimes R_\mathrm{sgn} \cdot \psi
    \end{array}\]
    and the $\S$-bimodule of relations $R^\perp_0$ (``orthogonal complement in genus zero'') is generated by 
    
    \medskip
    \tiny 
    $ (1) \qquad \tikzfig{
        \node [vertex] (bot) at (0,-0.2) {$\nu$};
        \node [vertex] (top) at (-0.5,0.5) {$\nu$};
        \draw [->-] (-1,1) to (top);
        \draw [->-] (0,1) to (top);
        \draw [->-] (top) to (bot);
        \draw [->-] (1,1) to (bot);
        \draw [->-] (bot) to (0,-1);
    } ~ + ~ \tikzfig{
        \node [vertex] (bot) at (0,-0.2) {$\nu$};
        \node [vertex] (top) at (0.5,0.5) {$\nu$};
        \draw [->-] (-1,1) to (bot);
        \draw [->-] (0,1) to (top);
        \draw [->-] (top) to (bot);
        \draw [->-] (1,1) to (top);
        \draw [->-] (bot) to (0,-1);
    } \qquad \qquad   \qquad (2)  \qquad \tikzfig{
        \node [vertex] (top) at (0,0.5) {$\nu$};
        \node [vertex] (bot) at (0,-0.5) {$\omega$};
        \draw [->-] (-0.6,1) to (top);
        \draw [->-] (top) to (bot);
        \draw [->-] (0.6,1) to (top);
        \draw [-w-] (bot) to (-0.6,-1);
        \draw [->-] (bot) to (0.6,-1);
    } ~ + (-1)^{n-1} ~\tikzfig{
        \node [vertex] (top) at (0.6,0) {$\omega$};
        \node [vertex] (bot) at (-0.6,0) {$\nu$};
        \draw [->-] (-0.6,1) to (bot);
        \draw [-w-] (top) to (bot);
        \draw [->-] (0.6,1) to (top);
        \draw [->-] (bot) to (-0.6,-1);
        \draw [->-] (top) to (0.6,-1);
    }  $ 
    
     $ (3) \qquad  \tikzfig{
        \node [vertex] (top) at (0.6,0) {$\omega$};
        \node [vertex] (bot) at (-0.6,0) {$\nu$};
        \draw [->-] (-0.6,1) to (bot);
        \draw [-w-] (top) to (bot);
        \draw [->-] (0.6,1) to (top);
        \draw [->-] (bot) to (-0.6,-1);
        \draw [->-] (top) to (0.6,-1);
    } ~ - ~ \tikzfig{
        \node [vertex] (top) at (-0.6,0) {$\omega$};
        \node [vertex] (bot) at (0.6,0) {$\nu$};
        \draw [->-] (-0.6,1) to (top);
        \draw [->-] (top) to (bot);
        \draw [->-] (0.6,1) to (bot);
        \draw [-w-] (top) to (-0.6,-1);
        \draw [->-] (bot) to (0.6,-1);
    }\qquad  \qquad \qquad \qquad \; \qquad(4)  \qquad  \tikzfig{
        \node [vertex] (left) at (-0.5,0) {$\omega$};
        \node [vertex] (right) at (0.5,0) {$\nu$};
        \draw [->-] (0,1) to (left);
        \draw [->-] (0,-1) to (right);
        \draw [-w-] (left) to (-1.2,0);
        \draw [->-] (left) to (right);
        \draw [->-] (right) to (1.2,0);
    } ~ + ~ \tikzfig{
        \node [vertex] (left) at (-0.5,0) {$\nu$};
        \node [vertex] (right) at (0.5,0) {$\omega$};
        \draw [->-] (0,1) to (right);
        \draw [->-] (0,-1) to (left);
        \draw [->-] (left) to (-1.2,0);
        \draw [-w-] (right) to (left);
        \draw [->-] (right) to (1.2,0);
    } $   

    $  (5) \qquad    \tikzfig{
        \node [vertex] (left) at (-0.5,0) {$\nu$};
        \node [vertex] (right) at (0.5,0) {$\omega$};
        \draw [->-] (0,1) to (right);
        \draw [->-] (0,-1) to (left);
        \draw [->-] (left) to (-1.2,0);
        \draw [-w-] (right) to (left);
        \draw [->-] (right) to (1.2,0);
    } ~ + (-1)^n \tikzfig{
        \node [vertex] (left) at (-0.5,0) {$\omega$};
        \node [vertex] (right) at (0.5,0) {$\nu$};
        \draw [->-] (0,1) to (right);
        \draw [->-] (0,-1) to (left);
        \draw [->-] (left) to (-1.2,0);
        \draw [-w-] (left) to (right);
        \draw [->-] (right) to (1.2,0);
    } \; \qquad (6)  \qquad   \tikzfig{
        \node [vertex] (top) at (0,0.4) {$\psi$};
        \node [vertex] (bot) at (0,-0.6) {$\nu$};
        \draw [-w-] (bot) to (top);
        \draw [->-] (bot) to (-1.2,-1);
        \draw [->-] (bot) to (1.2,-1);
        \draw [->-] (-1.2,0.4) to (top);
        \draw [->-] (top) to (0,1);
    } ~ + \qquad \tikzfig{
        \node [vertex] (top) at (0,0.1) {$\psi$};
        \node [vertex] (bot) at (-0.7,-0.5) {$\nu$};
        \draw [-w-] (top) to (0,1);
        \draw [->-] (bot) to (-1.2,-1);
        \draw [->-] (top) to (1.2,-1);
        \draw [->-] (-1.2,0.4) to (bot);
        \draw [->-] (top) to (bot);
    } $
    
     $  (7) \qquad  \tikzfig{
        \node [vertex] (top) at (0,0.1) {$\psi$};
        \node [vertex] (bot) at (-0.7,-0.5) {$\nu$};
        \draw [-w-] (top) to (0,1);
        \draw [->-] (bot) to (-1.2,-1);
        \draw [->-] (top) to (1.2,-1);
        \draw [->-] (-1.2,0.4) to (bot);
        \draw [->-] (top) to (bot);
    } ~ + (-1)^n ~ \tikzfig{
        \node [vertex] (top) at (0,0.4) {$\omega$};
        \node [vertex] (bot) at (0,-0.6) {$\omega$};
        \node at (0.4,0.4) {$\scriptstyle{1}$};
        \node at (0.4,-0.4) {$\scriptstyle{2}$};
        \draw [-w-] (top) to (bot);
        \draw [-w-] (bot) to (-1.2,-1);
        \draw [->-] (bot) to (1.2,-1);
        \draw [->-] (-1.2,0.4) to (top);
        \draw [->-] (top) to (0,1);
    }  \qquad (8)  \qquad   \tikzfig{
        \node [vertex] (top) at (0,0.4) {$\omega$};
        \node [vertex] (bot) at (0,-0.6) {$\omega$};
        \node at (0.4,0.4) {$\scriptstyle{1}$};
        \node at (0.4,-0.4) {$\scriptstyle{2}$};
        \draw [-w-] (top) to (bot);
        \draw [-w-] (bot) to (-1.2,-1);
        \draw [->-] (bot) to (1.2,-1);
        \draw [->-] (-1.2,0.4) to (top);
        \draw [->-] (top) to (0,1);
    } ~ + (-1)^{n-1} ~ \tikzfig{
        \node [vertex] (top) at (0,0.1) {$\omega$};
        \node [vertex] (bot) at (-0.7,-0.5) {$\omega$};
        \node at (0.4,0.1) {$\scriptstyle{1}$};
        \node at (-0.3,-0.5) {$\scriptstyle{2}$};
        \draw [->-] (top) to (0,1);
        \draw [-w-] (bot) to (-1.2,-1);
        \draw [-w-] (top) to (1.2,-1);
        \draw [->-] (-1.2,0.4) to (bot);
        \draw [->-] (bot) to (top);
    } $
    
 $ (9) \qquad
    \tikzfig{
        \node [vertex] (top) at (0,0.5) {$\omega$};
        \node [vertex] (bot) at (0,-0.5) {$\psi$};
        \draw [-w-] (bot) to (top);
        \draw [->-] (top) to (-1,1);
        \draw [-w-] (top) to (1,1);
        \draw [->-] (bot) to (-1,-1);
        \draw [->-] (bot) to (1,-1);
    } ~ + (-1)^{n-1} ~ \tikzfig{
        \node [vertex] (left) at (-0.5,0) {$\omega$};
        \node [vertex] (right) at (0.5,0) {$\psi$};
        \draw [-w-] (right) to (left);
        \draw [-w-] (left) to (-1,1);
        \draw [->-] (left) to (-1,-1);
        \draw [->-] (right) to (1,-1);
        \draw [->-] (right) to (1,1);
    } ~. $

\normalsize

\begin{proposition}\label{prop:subdioperad}
    There is an injective map of dioperads
    \[ i \colon \calY^{(n)!} \hookrightarrow \calV^{(n)!} \]
    whose image is exactly the subdioperad spanned by operations of arity different from $(2;0)$.
\end{proposition}
\begin{proof}
    We use the description of $\calV^{(n)!}$ from \cite[Proposition~3.7]{LV22}. It has two generators, $\pi$ in arity $(1;2)$ and $\gamma$ in arity $(2;0)$. We give the map $i$ by describing it on generators: \tiny
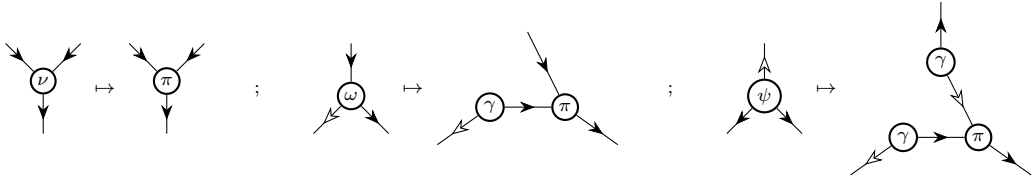
    \[
    \tikzfig{
        \node[vertex] (nu) at (0,0) {$\nu$};
        \draw [->-] (-0.5,0.5) to (nu);
        \draw [->-] (0.5,0.5)to (nu);
        \draw [->-] (nu) to (0,-0.7);
    } ~ \mapsto~
    \tikzfig{
        \node[vertex] (nu) at (0,0) {$\pi$};
        \draw [->-] (-0.5,0.5) to (nu);
        \draw [->-] (0.5,0.5)to (nu);
        \draw [->-] (nu) to (0,-0.7);
    } ~ \qquad ; \qquad ~
    \tikzfig{
        \node[vertex] (a) at (0,0) {$\omega$};
        \draw [-w-] (a) to (-0.5,-0.5);
        \draw [->-] (a) to (0.5,-0.5);
        \draw [->-] (0,0.7) to (a);
    } ~ \mapsto~
    \tikzfig{
        \node [vertex] (gamma) at (-0.5,0) {$\gamma$};
        \node [vertex] (pi) at (0.5,0) {$\pi$};
        \draw [-w-] (gamma) to (-1.2,-0.5);
        \draw [->-] (gamma) to (pi);
        \draw [->-] (pi) to (1.2,-0.5);
        \draw [->-] (0,1) to (pi);
    } ~ \qquad ; \qquad ~
    \tikzfig{
        \node[vertex] (b) at (0,0) {$\psi$};
        \draw [->-] (b) to (-0.5,-0.5);
        \draw [->-] (b) to (0.5,-0.5);
        \draw [-w-] (b) to (0,0.7);
    } ~ \mapsto~
    \tikzfig{
        \node [vertex] (left) at (-0.5,0) {$\gamma$};
        \node [vertex] (right) at (0.5,0) {$\pi$};
        \node [vertex] (top) at (0,1) {$\gamma$};
        \draw [-w-] (left) to (-1.2,-0.5);
        \draw [->-] (left) to (right);
        \draw [->-] (right) to (1.2,-0.5);
        \draw [-w-] (top) to (right);
        \draw [->-] (top) to (0,1.8);
    }~.
    \]
    \normalsize
    One can check all the necessary relations to see that this gives a map of dioperads. As for the injectivity and the characterization of the image of $i$, we can prove both statements by giving a basis of $\calY^{(n)!}$ which maps bijectively to the appropriate subset of a basis of $\calV^{(n)!}$.
    
    For that, we rely on the planar symmetry that is made apparent by the diagrammatics. By definition, the dioperad $\calY^{(n)!}$ is spanned by \emph{directed trees} whose internal vertices all have total arity three, and have as arity $(1;2),(2;1)$ or $(3;0)$, labeled respectively with $\nu,\omega$ and $\psi$. Every such directed tree can be embedded in the disc, such that all these vertices are in the orientation shown above; to see this, just pick any embedding and for every vertex oriented incorrectly we exchange the two branches of the tree coming out of it, putting it in the desired orientation. Thus, for each such directed tree of arity $(\ell;N)$, we can read the inputs and outputs around the disc, uniquely giving a sequence
    \[ S = \left(\sigma_1, (\tau_1,\dots, \tau_{k_1}), \sigma_2, (\tau_{k_1+1},\dots,\tau_{k_1+k_2}),\dots,\sigma_\ell, (\tau_{k_1+\dots+k_{\ell-1}},\dots,\tau_{k_1+\dots+k_{\ell}}) \right)\]
    where $N = k_1 + \dots +k_\ell$. Here $\sigma$ is the permutation giving the ordering of outputs around the boundary circle and $\tau$ is the permutation of $N$ giving the ordering of inputs; there are $k_j$ inputs in between the $j$th output $j$ and the $(j+1)$th output. 
    
    We note now that the relations of the dioperad $\calY^{(n)!}$ say exactly that, up to sign, every such directed tree with fixed sequence $S$ gets identified; this is because every directed tree with trivalent internal vertices can be transformed into any other such directed tree by homotoping through tetravalent vertices, and the relations say exactly that these moves act by a sign. Therefore $\calY^{(n)!}$ has one basis element for each sequence $S$; for each basis element we can find a tree that is in standard shape, analogous to the shape given in \cite[Equation~26]{LV22}. We deduce that $i$ gives an identification between this basis and the subset of that basis for $(\calV_{(n)})^!$ that is missing exactly the arity $(2;0)$ element.
\end{proof}

\begin{remark}
    We note that the map $i$ is a map of dioperads, but not of \emph{quadratic dioperads}, in the sense that it does not come from a map of the generators for which the relations are quadratic. Thus it cannot be carried through quadratic duality; there is no map $\calV^{(n)} \to \calY^{(n)}$ which corresponds to it, which is obvious since any such map would kill the copairing in $\calV^{(n)}$ and therefore factor through the associative operad.
\end{remark}

\begin{proposition}\label{prop:isomorphismCodioperads}
    There is an isomorphism of codioperads $\calC \cong \calY^{(n)\antish}$, inducing a bijection between $\calY^{(n)}_\infty$-structures and pre-CY-structures with vanishing copairing.
\end{proposition}
\begin{proof}
    The isomorphism is given by the dual map $i^\vee$ to the map in \cref{prop:subdioperad}, which gives an isomorphism to the quotient $\calC$ in this finitely-generated setting. We can thus identify the quotients. Applying cobar, we get an isomorphism of codioperads between $\calY^{(n)}_\infty$ and $\Omega\calC$.
\end{proof}

\subsection{Dioperads vs properads}
So far in this section, the discussion has been entirely about dioperads and codioperads, that is, objects that encode (de)composition maps indexed by trees. We follow the discussion of the relations between dioperads and properads in \cite[Section~5.6]{Merkulov_2009}. There are adjoint functors
\[ F \colon \mathrm{Dioperads} \leftrightarrows \mathrm{Properads} \colon U \]
between the categories of dioperads and properads. The forgetful functor $U$ preserves the underlying $\S$-bimodule and remembers only the genus zero compositions. Its left adjoint $F$ freely adjoins higher genus compositions without any higher genus relations. As for codioperads and coproperads, every codioperad is also a coproperad: there is an inclusion
\[ I \colon \mathrm{Codioperads} \hookrightarrow \mathrm{Coproperads} \]
preserving the $\S$-bimodules. A coproperad $\calC$ is in the image of $I$ if and only if its decomposition map lies in the genus zero component. In that case, we will simply say that $\calC$ \emph{is a codioperad.}

In general, given a quadratic dioperad $\calQ$, the quadratic dual \emph{coproperad} $(F\calQ)^\antish$ will not be in the image of $I$, since its decomposition map will not only generate genus zero graphs. We can characterize when that \emph{is} the case by linear duality. For that, note that we have the following maps of dioperads
\[\begin{array}{ccccc}
    \calQ^! & \to   & UF(\calQ^!) & \to & U((F\calQ)^!) \\
    \verteq &   & \verteq &     & \verteq \\
    \calT(s^{-1} E^\vee)/\langle s^{-2} R^\perp_{0} \rangle & & \calG(s^{-1}E^\vee)/\langle s^{-2}R^\perp_{0} \rangle & & \calG(s^{-1}E^\vee)/\langle s^{-2}R^\perp \rangle
\end{array}\]
where $R^\perp_0$ is the `dioperadic' orthogonal complement to $R$, a subspace of $\calT(E)^{(2)}$, and $R^\perp$ is its `properadic' orthogonal complement, a subspace of $\calG(E)^{(2)}$. We rephrase the statements in  \cite{Merkulov_2009} in the following way:
\begin{proposition}\label{prop:contractible}
    Let $\calQ$ be a finitely-generated quadratic dioperad. Then $(F\calQ)^\antish$ is a codioperad if and only if the composition
    \[ \calQ^! \to UF(\calQ^!) \to U((F\calQ)^!) \]
    is an isomorphism of dioperads. In that case, the canonical map of dg properads
    \[ F(\Omega \calQ^\antish) \to \Omega(F\calQ)^\antish \]
    is an isomorphism.
\end{proposition}
\begin{proof}
    The first statement follows from finite generation by linear duality, and the second statement follows from applying \cite[Proposition~44]{Merkulov_2009} to calculate that the underlying dg $\S$-bimodule of $F(\Omega \calQ^\antish)$ is exactly given by the cobar construction of $\calQ^\antish$; under the assumption we get the equality to $\Omega((F\calQ)^\antish)$, which implies the equivalence of categories.
\end{proof}

\begin{corollary}
    If $\calQ$ is a finitely-generated quadratic dioperad such that $(F\calQ)^\antish$ is a codioperad, then given any chain complex $A$, there is an isomorphism
    \[ \g^{\calQ,\mathrm{diop}}_{(A)} = \left(\Hom_{\S \otimes \S^\mathrm{op}}(\calQ^\antish,\End_A), \partial, \star \right) \cong \left(\Hom_{\S \otimes \S^\mathrm{op}}((F\calQ)^\antish,\End_A), \partial, \star \right) = \g^{\calQ,\mathrm{prop}}_{(A)} \]
    between the dg Lie-admissible algebras controlling (dioperadic) $Q$-algebras and (properadic) $F\calQ$-algebras.
\end{corollary}

\noindent Following the argument at the end of \cite{poirier2019koszuality}, it is proven in \cite{LV22} that the dioperad $\calV^{(n)}$ \emph{does not} satisfy the condition of \cref{prop:contractible}; in other words, the notions of properadic (i.e. all-genus) $\calV^{(n)}_\infty$-algebra and dioperadic (i.e. genus zero) $\calV^{(n)}_\infty$-algebra are genuinely different. In contrast, there are two dioperads related to $\calV^{(n)}$ that are known to satisfy the condition of \cref{prop:contractible}:
\begin{itemize}
    \item[$\centerdot$] the dioperad $\mathrm{DPois}$ governing double Poisson algebras \cite{LV22}, and
    \item[$\centerdot$] the dioperad $\mathrm{BIB}^\lambda$ governing (shifted) balanced infinitesimal bialgebras \cite{quesney2023balanced}.
\end{itemize}

\noindent Our dioperad $\calY^{(n)}$ is very close to what is denoted in \emph{op.cit.} by $\mathrm{BIB}^{1-n}$; the only difference is that we also have a generator in arity $(3,0)$, whose presence modifies the relations between the other generators. In other words, there is a quotient of quadratic dioperads $\calY^{(n)} \twoheadrightarrow \mathrm{BIB}^{1-n}$ killing the generator of arity $(3,0)$. We use this fact to prove that the condition of \cref{prop:contractible} is also satisfied by $\calY^{(n)}$. Let us first give an explicit presentation of the quadratic dual properad.
\begin{proposition}
    The quadratic dual properad of $F\calY^{(n)}$ is given by 
    \[ \left(F \calY^{(n)}\right)^! = \calG(s^{-1} E^\vee)/\langle s^{-2} R^\perp \rangle \] 
    where $R^\perp$ is generated by $R^\perp_0$ together with all the quadratic higher genus elements, for which the following four elements
    \begin{align*}
        \tikzfig{
            \node [vertex] (left) at (-1,0) {$\omega$};
            \node [vertex] (right) at (0,0) {$\nu$};
            \draw [->-] (-1.8,0) to (left);
            \draw [-w-,bend right=60] (left) to (right);
            \draw [->-,bend left=60] (left) to (right);
            \draw [->-] (right) to (0.8,0);
        } ~ ; ~ 
        \tikzfig{
            \node [vertex] (left) at (-1,0) {$\omega$};
            \node [vertex] (right) at (1,0) {$\nu$};
            \draw [->-] (-1.8,0) to (left);
            \draw [-w-,bend right=60] (left) to (0,0);
            \draw [->-,bend left=60] (left) to (0,0);
            \draw [bend right=60] (0,0) to (right); 
            \draw [bend left=60] (0,0) to (right); 
            \draw [->-] (right) to (1.8,0);
        } ~ ; ~
        \tikzfig{
            \node [vertex] (left) at (-1,0) {$\psi$};
            \node [vertex] (right) at (0,0) {$\nu$};
            \draw [->-] (left) to (-1.8,0);
            \draw [-w-,bend right=60] (left) to (right);
            \draw [->-,bend left=60] (left) to (right);
            \draw [->-] (right) to (0.8,0);
        } ~ ; ~ 
        \tikzfig{
            \node [vertex] (left) at (-1,0) {$\psi$};
            \node [vertex] (right) at (1,0) {$\nu$};
            \draw [->-] (left) to (-1.8,0);
            \draw [-w-,bend right=60] (left) to (0,0);
            \draw [->-,bend left=60] (left) to (0,0);
            \draw [bend right=60] (0,0) to (right); 
            \draw [bend left=60] (0,0) to (right); 
            \draw [->-] (right) to (1.8,0);
        }
    \end{align*}
    are generators.
\end{proposition}
\begin{proof}
    Since the relations of $\calY^{(n)}$ are all in genus zero, all the quadratic higher genus relations are in the orthogonal complement. Then, the four elements above generate all of those by symmetry.
\end{proof}

\begin{theorem}\label{thm:codioperad}
    The coproperad $(F\calY^{(n)})^\antish$ is a codioperad.
\end{theorem}
\begin{proof}
    Let us fix $n$ and denote $\mathrm{BIB} = \mathrm{BIB}^{1-n}$. By \cref{prop:contractible} it is sufficient to prove that $(F\calY^{(n)})^!$ is trivial in higher genus, or in other words, that the submodule $\langle R^\perp \rangle$ contains every element given by a graph of genus $\geqslant 1$. We already know that $R^\perp$ contains every higher-genus graph with exactly two internal vertices, so we consider from now on graphs that have at least three internal vertices. Since $\calY^{(n)!}$ is generated by $\mathrm{BIB}^!$ and an extra generator $\psi$, every such graph is given by attaching $\psi$ vertices, of arity $(3,0)$, to collections of $\mathrm{BIB}^!$-graphs, such as:
    \[\tikzfig{
        \node [vertex,inner sep=8pt] (bib1) at (-1.73,1) {$\mathrm{BIB}^!$};
        \node [vertex,inner sep=8pt] (bib2) at (1.73,1) {$\mathrm{BIB}^!$};
        \node [vertex,inner sep=8pt] (bib3) at (0,-2) {$\mathrm{BIB}^!$};
        \node [vertex] (psi1) at (0,1.73) {$\psi$};
        \node [vertex] (psi2) at (1,-0.6) {$\psi$};
        \node [vertex] (psi3) at (-1,-0.6) {$\psi$};
        \draw [->-] (-3.2,1) to (bib1);
        \draw [->-] (bib1) to (-3,0);
        \draw [->-] (3.2,1) to (bib2);
        \draw [->-] (bib3) to (-1.5,-2);
        \draw [->-] (bib3) to (1.5,-2);
        \draw [-w-] (psi1) to (bib1);
        \draw [->-,bend right=60] (psi1) to (bib2);
        \draw [->-,bend left=60] (psi1) to (bib2);
        \draw [->-] (psi2) to (bib2);
        \draw [->-] (psi2) to (bib2);
        \draw [->-] (psi2) to (2,-0.6);
        \draw [-w-] (psi2) to (bib3);
        \draw [->-] (psi3) to (bib1);
        \draw [->-] (psi3) to (-1.8,-1.4);
        \draw [-w-] (psi3) to (bib3);
    }\]
    By \cite[Corollary~24]{quesney2023balanced}, every $\mathrm{BIB}^!$-graph with genus $\geqslant 1$ is zero, so it is enough to look at diagrams as above where each $\mathrm{BIB}^!$-subgraph is a tree, which we can put in planar form. If the graph has nonzero genus, there must be a cycle with $k \geqslant 1$ vertices $\psi$. %
    The $k$ strands connecting these $\psi$ vertices have sequences of $\omega$ and $\nu$ vertices. Since there can be no internal sink or vertices, that is, vertices with no outgoing edges, each of the $k$ strands must have at least one $\nu$ vertex. We then use the relations labeled (2) to (5) in our calculation of $\calY^{(n)!}$  above to `bring' each such $\nu$ vertex to be adjacent to a $\psi$ vertex, getting a cycle that looks like
    \[
    \tikzfig{
        \node [vertex] (psi1) at (0,2.7) {$\psi$};
        \node [vertex] (nu1) at (0.9,1.8) {$\nu$};
        \node [vertex] (psi2) at (2.7,0) {$\psi$};
        \node [vertex] (nu2) at (1.8,-0.9) {$\nu$};
        \node [vertex] (psi3) at (-2.7,0) {$\psi$};
        \node [vertex] (nu3) at (-1.8,0.9) {$\nu$};
        \node (dots1) at (1.8,0.9) {$\rotatebox[origin=c]{-45}{\dots}$};
        \node (dots2a) at (0,-0.9) {$\dots$};
        \node (dots2b) at (-1.5,-0.9) {$\dots$};
        \node (dots3) at (-0.9,1.8) {$\rotatebox[origin=c]{45}{\dots}$};
        \draw [->-] (psi1) to (nu1);
        \draw [->-] (nu1) to (dots1);
        \draw [->-] (psi1) to (dots3);
        \draw [->-] (psi2) to (nu2);
        \draw [->-] (nu2) to (dots2a);
        \draw [->-] (psi2) to (dots1);
        \draw [->-] (psi3) to (nu3);
        \draw [->-] (nu3) to (dots3);
        \draw [->-] (psi3) to (dots2b.180);
        \draw [->-] (1.5,2.4) to (nu1);
        \draw [->-] (1.2,-0.3) to (nu2);
        \draw [->-] (-2.4,1.5) to (nu3);
        \draw [->-] (psi1) to (-1,2.7);
        \draw [->-] (psi2) to (3.3,-0.6);
        \draw [->-] (psi3) to (-3.3,-0.6);
    }\]
    Note that this graph may not be embedded in the plane as in the example above. In any case, we can use relation (7) to eliminate all $\psi$ vertices from the cycle. Repeating the procedure until all cycles are gone, we get to a single $\mathrm{BIB}^!$-graph, possibly with $\psi$ vertices attached to it by a single edge, that is, not forming any cycles of the form above. Note that none of the relations changes the topology of the underlying graph, therefore the resulting $\mathrm{BIB}^!$-graph has the same genus $\geqslant 1$ we started with, and vanishes by the aforementioned result.
\end{proof}

\begin{remark}
   By \cref{prop:contractible} and \cref{thm:codioperad}, we have equivalent categories of (dioperadic) $\calY^{(n)}_\infty$-algebras and (properadic) $(F\calY^{(n)})_\infty$-algebras. We will, from now on, elide this distinction and just use $\calY^{(n)}_\infty$ to also refer to the properad, and work purely in the setting of properadic algebras. 
\end{remark}

\subsection{Pre-CY and $\calY^{(n)}_\infty$-structures on based loop spaces}\label{subsec:YinfinityBasedLoop}
Let us describe the class of $\calY^{(n)}$-algebras that will be our main source of examples. For that, we recall the relation between pre-Calabi--Yau and Calabi--Yau structures from \cite{KTV2}. Let $(A,\mu)$ be an $A_\infty$-algebra over $R$, which is \emph{smooth}, in the sense that its diagonal bimodule $A$ is perfect; in that case, we have a quasi-isomorphism
\[ CH_*(A) \xrightarrow{\simeq} \R\Hom^*_{A-A}(A^!,A) \]
between Hochschild chains and the derived bimodule morphisms between the inverse dualizing bimodule $A^!$ and the diagonal bimodule $A$. We recall that there is a canonical map $HC^-_*(A) \to HH_*(A)$ from negative cyclic to Hochschild homology.
\begin{definition}
    A (smooth) \emph{$n$-Calabi--Yau structure} on $A$ is a negative cyclic homology class 
    \[[\omega] \in HC^-_n(A)\] 
    whose image under the map above gives a quasi-isomorphism of bimodules $A^! \xrightarrow{\sim} A$.
\end{definition}

\noindent One should regard a smooth $n$-Calabi--Yau structure as something like a noncommutative $n$-shifted symplectic form. We now paraphrase a result of \cite{KTV2}, which upgrades the usual relation between symplectic forms and nondegenerate Poisson bivectors to this noncommutative level. The role of Poisson bivector is played by an element of $CH^*_{(2)}(A)$. If $\alpha$ is a closed element of degree $n$ in the complex \[(CH^*_{(2)}(A),[\mu,-]) \ ,\] we have a map of complexes
\[ g_\alpha \colon CH_{*}(A) \to CH^{-*}(A)[n] \]
given by evaluating a certain diagram, which is an analog of the map between forms and vector fields induced by a bivector.
\begin{definition}\label{def:compatibilityTwoOutputs}
    Let $A$ be a smooth $A_\infty$-algebra and $[\omega_0] \in HH_n(A)$ be a Hochschild homology class. We say that an element
    \[ \varphi_{(2)} \in CH^n_{(2)}(A) \]
    is \emph{compatible} with $[\omega_0]$ if it satisfies the equation $[g_{\varphi_{(2)}}(\omega_0)] \cong 1$ in $HH^0(A)$.
\end{definition}

\begin{theorem}[{\cite[Proposition~4.2 and Theorem~4.7]{KTV2}}]\label{thm:dualPreCY}
    For any smooth $A_\infty$-algebra $A$ and nondegenerate Hochschild homology class $[\omega_0] \in HH_n(A)$, there is a bijection between the following sets:
    \begin{enumerate}
        \item the set of lifts of $[\omega_0]$ to a smooth $n$-CY structure, that is, to a negative cyclic class $[\omega]$ in its preimage under the map $HC^-_*(A) \to HH_*(A)$, and 
        \item $\pi_0$ of the space of $n$-pre-CY structures $\varphi$ on $A$ whose component \[\varphi_{(2)} \in CH^n_{(2)}(A) \] is compatible with $\omega_0$. 
    \end{enumerate}
\end{theorem}

\begin{remark} 
    In the theorem above, it is essential that the space of pre-CY structures be modeled as a certain simplicial space. We note that in this context the notion of `equivalence', that is, being in the same connected component, is more general than the notion of gauge equivalence. For more details see the discussion in section 4 of \emph{op.cit.}
\end{remark}

\noindent One source of such smooth Calabi--Yau structures is the algebra of chains on based loop spaces of \emph{local Poincar\'e duality pairs}, that is, topological spaces with a fundamental class giving Poincar\'e duality with local system coefficients. Let $X$ be any path-connected space with the homotopy type of a finite simplicial complex, and denote $A = C_*(\Omega X,R)$; this is a dg algebra, whose product is induced by concatenation of based loops.
\begin{theorem}[{\cite{holstein2024koszul}}]\label{thm:holstein}
    Let $R$ be any ring and $[X] \in H_n(X,R)$ any degree $n$ class, represented by some $n$-chain $\alpha_X$. If $(X,[X])$ is a local Poincar\'e duality pair of dimension $n$, in the sense that $\frown \alpha_X$ induces a quasi-isomorphism
    \[ C^*(X,\calL) \xrightarrow{\sim} C_{n-*}(X,\calL) \]
    for any local system $\calL$, then $\alpha_X$ induces a smooth Calabi--Yau structure on $A = C_*(\Omega X,R)$. Moreover, the class of this smooth CY structure in negative cyclic homology only depends on $[X]$ and not on the chosen representative $\alpha_X$.
\end{theorem}

\begin{remark}
    The theorem above generalizes previous results, in the case of manifolds, of \cite{Abb15,cohencalabi}.
\end{remark}

\begin{definition}\label{def:compatibilityPreCY}
    Let $(X,[X])$ be a local Poincar\'e duality pair of dimension $n$. We say that an $n$-pre-CY algebra structure $\varphi$ on $A = C_*(\Omega X,R)$ is \emph{compatible} with $[X]$ if $\varphi$ corresponds to the smooth CY structure of \cref{thm:holstein} under the bijection of \cref{thm:dualPreCY}.
\end{definition}

\noindent Recall that $n$-pre-CY structures are algebra structures over the dg dioperad $\calV^{(n)}_\infty$, which is a cofibrant resolution of the dioperad $\calV^{(n)}$ of \cite{poirier2019koszuality}. As mentioned in the introduction, we use the notion of formality discussed in \cite[Section~2]{Kaledin}. If $\calC$ is a weight-graded reduced cooperad/codioperad/coproperad, there is a morphism of operads/dioperads/properads
\[ \Omega\calC \to \calP = \left. \calT(s^{-1}\calC^{(1)}) \middle/ d_{\Omega\calC}(s^{-1}\calC^{(2)}) \right. \]
and any $\Omega\calC$-algebra structure on a complex induces a $\calP$-structure on its homology. One can define formality of such structures in the following sense.
\begin{definition}\cite[Definition~2.5]{Kaledin}\label{def:formality}
    A $\Omega\calC$-algebra $(A,\varphi)$ is said to be \emph{formal} if there is a zig-zag of quasi-isomorphisms of $\Omega\calC$-algebras
    \[ (A, \varphi) \overset{\sim}{\longleftarrow} \cdot \overset{\sim}{\longrightarrow} \cdots \overset{\sim}{\longleftarrow} \cdot \overset{\sim}{\longrightarrow} (H_*(A), \varphi_*) \ . \]
    A graded $\calP$-algebra $(H,\varphi_*)$ is said to be \emph{intrinsically formal} if, for any chain complex $A$ with $H_*(A)=H$ and $\Omega\calC$-algebra structure $\varphi$ on $A$ inducing $\varphi_*$ on $H$, the $\Omega\calC$-algebra $(A,\varphi)$ is formal.
\end{definition}

\noindent We could consider the setting where the morphism given by the cofibrant resolution of dioperads
\[ \calV^{(n)}_\infty = \Omega(\calV^{(n)})^\antish \to \calV^{(n)}. \]
However, it turns out that for the class of algebras we are interested in, the question of formality as $\calV^{(n)}_\infty$-algebras is not interesting, due to the following result.
\begin{theorem}\label{thm:notVinftyFormal}
    Let $(X,[X])$ be a local Poincar\'e duality pair of dimension $n \ge 1$. None of the $n$-pre-CY algebra structures on $A = C_*(\Omega X,R)$ that are compatible with $[X]$ are formal.
\end{theorem}
\begin{proof}
    Consider any $n$-pre-CY whose $m_{(2)}$ component is compatible with the smooth CY structure. Compatibility implies that the class of $m_{(2)}$ is nonzero in $HH^n_{(2)}(A)$. Suppose for the sake of contradiction that $(A,m)$ is formal as a $\calV^{(n)}_\infty$-algebra, that is, that there is a zig-zag of $\calV^{(n)}_\infty$-quasi-isomorphisms between $(A,m)$ and its homology with induced $\calV^{(n)}$-structure $(H,n)$. In particular, this implies that 
    \begin{enumerate}
        \item $A$ is formal as an $A_\infty$-algebra, by giving a zig-zag of $A_\infty$-quasi-isomorphisms between $(A,m_{(1)})$ and $(H,n_{(1)})$, and
        \item the classes $[m_{(2)}] \in HH^n_{(2)}(A)$ and $[n_{(2)}] \in HH^n_{(2)}(H)$ are identified by the isomorphism on Hochschild homology induced by this zig-zag.
    \end{enumerate}
    However, since $n \ge 1$ and $H$ is concentrated in non-negative homological degree, the element $n_{(2)}$ vanishes, contradicting the fact that $[m_{(2)}]$ is nonzero.
\end{proof}

\noindent As a result, it would not make much sense to define coformality of local Poincar\'e pairs in terms of $\calV^{(n)}_\infty$-algebra structures. We propose the setting of $\calY^{(n)}_\infty$-algebra structures as a better setting for this definition. We first identify which structures to associate to a pair.
\begin{definition}\label{def:compatibilityYinfinity}
    Let $(X,[X])$ be a local Poincar\'e pair. A $\calY^{(n)}_\infty$-algebra structure $\varphi$ on $A = C_*(\Omega X,R)$ is \emph{compatible} with $[X]$ if it corresponds under \cref{prop:isomorphismCodioperads} to a compatible $n$-pre-CY structure. Likewise, a $\calY^{(n)}$-algebra structure on $H_*(A) = H_*(\Omega X,R)$ is \emph{compatible} with $[X]$ if it is induced by a compatible $\calY^{(n)}_\infty$-algebra structure on $A$.
\end{definition}

\begin{proposition}\label{prop:existenceCompatible}
    Let $(X,[X])$ be a local Poincar\'e duality pair of dimension $n \geqslant 1$. Then there exist $\calY^{(n)}_\infty$-algebra structures on $A = C_*(\Omega X,R)$ that are compatible with $[X]$.
\end{proposition}
\begin{proof}
    This proposition follows from \cref{thm:dualPreCY} and \cref{prop:isomorphismCodioperads}.
\end{proof}

\noindent We now define coformality of a local Poincar\'e pair in terms of $\calY^{(n)}_\infty$-algebra structures.
\begin{definition}\label{def:coformality}
    Let $(X,[X])$ be a local Poincar\'e duality pair of degree $n \geqslant 1$.  We say that the pair $(X,[X])$ is
    \begin{itemize}
        \item \emph{coformal} over $R$ when all $\calY^{(n)}_\infty$-algebra structures on $A = C_*(\Omega X,R)$ that are compatible with $[X]$ are formal as $\calY^{(n)}_\infty$-algebras, and
        \item \emph{intrinsically coformal} over $R$ when all $\calY^{(n)}$-algebra structures on $H_*(A) = H_*(\Omega X,R)$ that are compatible with $[X]$ are intrinsically formal as $\calY^{(n)}$-algebras.
    \end{itemize} 
\end{definition}

\noindent In fact, when assessing intrinsic coformality, it will often be sufficient for us to study a single $\calY^{(n)}$-algebra structure, due to the following fact.
\begin{proposition}\label{prop:compatibleUniqueness}
    Let $n \ge 2$ and $\varphi$ and $\varphi'$ be two $\calY^{(n)}_\infty$-algebra structures on $A = C_*(\Omega X,R)$ that are compatible with $[X]$. Then their induced $\calY^{(n)}$-algebra structures $\varphi_*$ and $\varphi'_*$ on $H_*(A) = H_*(\Omega X,R)$ are equal.
\end{proposition}
\begin{proof}
    We note that the $\calY^{(n)}$-algebra structure on $H_*(A)$ induced by any $\calY^{(n)}_\infty$-algebra structure $\varphi$ on $A$ only depends on the component $(\varphi_{(2)})^{1,0}$ with one input and two outputs. If $\varphi$ and $\varphi'$ are any two compatible $\calY^{(n)}_\infty$-algebra structures, their components $\varphi_{(2)}$ and $\varphi'_{(2)}$ must be compatible with the same class in $HH_n(A)$; so by \cite[Proposition~2.2]{KTV2} they represent the same class in $HH^n_{(2)}(A)$. In other words, there is some element $\psi \in CH^{n-1}_{(2)}(A)$ such that
    \[ \varphi'_{(2)} - \varphi_{(2)} = d\psi \]
    However, since $n \geqslant 2$ and $A$ is supported in non-negative degrees, the component without outputs $\psi^{0,0}$ vanishes and $(\varphi_{(2)})^{1,0} = (\varphi'_{(2)})^{1,0}$.
\end{proof}

\begin{remark} 
    In \cref{prop:existenceCompatible}, the assumption of characteristic zero is necessary due to the inductive procedure used and the symmetrization at each step used in \cite{KTV2}. Moreover, over an arbitrary commutative ring $R$, even when there are compatible $\calY^{(n)}_\infty$-algebra structures on $A$, they might not induce $\calY^{(n)}$-algebra structures on homology, due to the failure of the K\"unneth map to be an isomorphism.
\end{remark}

Since the dg Lie algebra governing $A_\infty$-structures sits as a subalgebra of the one governing $\calY^{(n)}_\infty$-algebra structures, we have the following result.
\begin{proposition}
    Let $(X,[X])$ be a space with degree  $n \geqslant 1$ Poincar\'e duality with local coefficients that has an induced $\calY^{(n)}$-algebra structure $\varphi_*$ on $H = H_*(\Omega X,R)$.  If it is coformal, then $X$ is coformal in the sense that the dg algebra $C_*(\Omega X,R)$ is formal as an $A_\infty$-algebra.
\end{proposition}

\section{\textcolor{bordeau}{Intermediate obstruction sequences to formality}}\label{sec:intermediateObstruction}

\noindent In \cite{CE24b}, the first-named author constructs obstruction sequences to formality of algebras encoded by properads, such as pre-CY algebras with vanishing copairing. In \cref{sec:formalityLoopSpaces}, we will compute these obstructions to study pre-CY coformality of spheres. To this end, we will exploit the existence of an additional filtration; in this section, we explain the general framework for refining obstruction sequences in the presence of that extra filtration.

\subsection{Intermediate gauge triviality sequences}\label{subsec:intermediateGauge} 
We will work with the following assumptions.
\begin{assumptions}\label{ass:gradings} \leavevmode
    \begin{enumerate}
        \item Let $(\g,[-,-])$ be a weight-graded Lie algebra over $R$, i.e. a complete Lie algebra (with vanishing differential) with an additional weight grading such that \[\mathfrak{g} \cong \prod_{k \geqslant 1} \mathfrak{g}^{(k)},  \quad \left[ \mathfrak{g}^{(k)},  \mathfrak{g}^{(l)} \right] \subset  \mathfrak{g}^{(k +l)}. \]
        We denote by $\cdot$ the gauge group action and the canonical projections by
        \[\pi_k : \mathfrak{g} \twoheadrightarrow  \mathfrak{g} /\mathcal{F}^k \mathfrak{g},\ \text{where}\ \mathcal{F}^n \mathfrak{g} \coloneqq  \prod_{k \geqslant n} \mathfrak{g}^{(k)} .\]
        \item Suppose that $\g$ has an an additional descending filtration 
        \[\g = \mathcal{L}^0 \g \supset \mathcal{L}^1 \g \supset \mathcal{L}^2 \g \supset \cdots  \] that is compatible with the Lie bracket, in the sense that $[\calL^i \g, \calL^j \g] \subset \calL^{i+j} \g$ for any $i,j \geqslant 0$, and that is bounded with respect to $\calF$, that is, for every $k \geqslant 1$, there exists $\delta_k$ such that $\mathcal{L}^{\delta_k} \mathfrak{g} \subseteq \mathcal{F}^k \mathfrak{g}$.
    \end{enumerate}
\end{assumptions}

\noindent Let $\varphi$ and $\psi$ be two Maurer--Cartan elements in $\mathfrak{g}$, such that $\psi$ is of homogeneous weight one.  Let us set $\mathfrak{h}$ for the dg Lie algebra twisted by $\psi$
\[\mathfrak{h} \coloneqq \left(\mathfrak{g},  [-,-], d^{\psi} \coloneqq [\psi,- ], \mathcal{F} \right) \ .\] 
We want to detect whether $\varphi$ and $\psi$ are gauge equivalent, or equivalently, whether $\phi \coloneqq \varphi - \psi$ is \emph{gauge trivial}, i.e. if there exists $\lambda \in \mathfrak{h}_0$ such that $\lambda \cdot \phi = 0 $. By \cite[Section~1]{CE24b}, even without the additional filtration $\calL$, this can be achieved through the construction of a \emph{gauge triviality sequence} 
\[\lbrace \vartheta_{k} \rbrace_{1 \leqslant k \leqslant n} \ , \]
which is either an infinite sequence of vanishing homology classes, when $n = \infty$~; or a finite sequence of trivial classes that ends on a nonvanishing class $\vartheta_{n}$ for some $n \geqslant 1$. This index $n$ of the last class only depends on $\phi$ and is called the \emph{gauge triviality degree} of $\phi$ \cite[Theorem~1.17]{CE24b}. This gauge triviality sequence is constructed by induction: we first set 
\[\phi_1 \coloneqq \phi \quad \mbox{and} \quad \vartheta_1 \coloneqq \left[ \pi_2 (\phi_1)\right] \in H_{-1}\left( \mathfrak{h} / \mathcal{F}^{2} \mathfrak{h}  \right) \ . \]
Let us suppose that $\phi_i$ and $\vartheta_i$ have been constructed for all $1 \leqslant i \leqslant k$
\begin{itemize}
    \item[$\centerdot$] If $\vartheta_k \neq 0$~, then $k$ is the gauge triviality degree of $\phi$, and we stop.
    \item[$\centerdot$] If $\vartheta_k = 0$~, there must exist $\upsilon_k \in \mathfrak{h}$ such that $\upsilon_k \cdot \phi_k \in  \mathcal{F}^{k+1} \mathfrak{h}$. We set  
    \[\phi_{k+1} \coloneqq \upsilon_k \cdot \phi_k \quad \mbox{and} \quad \vartheta_{k+1} \coloneqq \left[ \pi_{k+2} (\phi_{k+1})\right] \in H_{-1}\left( \mathfrak{h} / \mathcal{F}^{k+2} \mathfrak{h}  \right) , \]
    and we can continue to the next value of $k$.
\end{itemize}

\noindent It can be difficult to determine in practice whether the class $\vartheta_k$ vanishes or not. This is where having an additional filtration will help. 

\begin{proposition}\label{prop:obstru}
    Under Assumptions \ref{ass:gradings}, let $\xi$ be a Maurer--Cartan element in $\h$  such that \[ \xi \in \calL^i \calF^k \h + \calF^{k+1} \h  \] for some $i \geqslant 0$ and $k \geqslant 2$. Let us consider its class in the homology of the quotient \[ \vartheta_{k}^i \coloneqq [\Imm(\xi)] \in H_{-1}\left(\frac{\h}{\calL^{i+1} \calF^k\h + \calF^{k+1}\h}\right). \]
    The following assertions are equivalent.
    \begin{enumerate}
        \item The homology class $\vartheta_{k}^i$ vanishes.
        \item There exists $\upsilon \in \mathfrak{h}_0$ such that $\upsilon \cdot \xi \in  \calL^{i+1} \calF^k \h + \calF^{k+1}\h$.
    \end{enumerate}
\end{proposition}

\begin{proof}
    For all $\upsilon \in \mathfrak{h}_0$, the gauge action formula gives \begin{equation}\label{eq}
        \upsilon \cdot \xi \equiv  \xi - \frac{ e^{ \mathrm{ad}_{\upsilon}} - \mathrm{id}}{\mathrm{ad}_{\upsilon}} (d^{\psi} \upsilon)  \pmod{\calL^{i+1} \calF^k\h + \calF^{k+1}\h} \ .
    \end{equation}  
    If $\vartheta_{k}^i = 0$, then there exists $\upsilon \in \mathfrak{h}_0$ such that \[ \xi \equiv d^{\psi} \upsilon  \pmod{\calL^{i+1} \calF^k\h + \calF^{k+1}\h} \ .\] Since $\xi \in \mathcal{F}^k \mathfrak{h}$, this implies that $d^{\psi} \upsilon \in \mathcal{F}^k \mathfrak{h}$ and  $\upsilon \in \mathcal{F}^{k-1} \mathfrak{h}$. Since $k \geqslant 2$, we have \[ \frac{ e^{ \mathrm{ad}_{\upsilon}} - \mathrm{id}}{\mathrm{ad}_{\upsilon}} (d^{\psi} \upsilon) \equiv d^{\psi} \upsilon \pmod{\calL^{i+1} \calF^k\h + \calF^{k+1}\h} \ . \] \cref{eq} implies that $\upsilon \cdot \xi \in  \calL^{i+1} \calF^k \h + \calF^{k+1}\h \ .$ Conversely, if point (2) holds, we have \[ \xi  \equiv \frac{ e^{ \mathrm{ad}_{\upsilon}} - \mathrm{id}}{\mathrm{ad}_{\upsilon}} (d^{\psi} \upsilon)   \pmod{\calL^{i+1} \calF^k\h + \calF^{k+1}\h} \]  by \cref{eq}. Since $\xi \in \mathcal{F}^k \mathfrak{h}$, this implies that $d^{\psi} \upsilon \in \mathcal{F}^k \mathfrak{h}$ and  \[\xi  \equiv d^{\psi} \upsilon \pmod{\calL^{i+1} \calF^k\h + \calF^{k+1}\h}  \ , \] which leads to $\vartheta_{k}^i = 0$. 
\end{proof}

\begin{remark}\label{vanishing}
    If the obstruction class $\vartheta_k^i$ is nonvanishing then so is the obstruction class $\vartheta_k$. Indeed, the  former is the image of the latter under the map
    \[ H_{-1}\left(\frac{\h}{\calF^{k+1}\h}\right) \to  H_{-1}\left(\frac{\h}{\calL^{i+1} \calF^k\h + \calF^{k+1}\h}\right). \]
\end{remark}

\begin{construction} \label{constru}
Let $\xi \in \mathrm{MC}(\calF^{k} \h)$ for $k \geqslant 2$.  We aim to detect whether there exists $\lambda \in \h_0$ such that $\lambda \cdot \xi \in   \mathcal{F}^{k +1} \mathfrak{h}$, or equivalently whether the obstruction class \[ \vartheta_k =  \left[ \pi_{k+2} (\phi_{k+1})\right] \in H_{-1}\left( \mathfrak{h} / \mathcal{F}^{k+2} \mathfrak{h}  \right) , \] vanishes, see \cite[Proposition~1.6]{CE24b}. Let us set $\xi_0 \coloneqq \xi $  and consider the first obstruction \[ \vartheta_{k}^0 \coloneqq [\Imm(\xi_0)] \in H_{-1}\left(\frac{\h}{\calL^{1} \calF^k\h + \calF^{k+1}\h}\right). \]
	\begin{itemize}
		\item[$\centerdot$] If $\vartheta_{k}^0$ is non-zero, then so is $\vartheta_k$.
		\item[$\centerdot$] If $\vartheta_{k}^0 = 0$~, there exists $\upsilon_0 \in \mathfrak{h}_0$~, such that $\upsilon_0 \cdot \xi_0 \in \calL^{1} \calF^k \h + \calF^{k+1}\h$, by the implication $(1) \Rightarrow (2)$ of \cref{prop:obstru}.
	\end{itemize}
	If $\vartheta_k^0 = 0$~, we set $\xi_1 \coloneqq \upsilon_0 \cdot \xi_0$ and \[\vartheta_k^1 \coloneqq [\Imm(\xi_1)] \in H_{-1}\left(\frac{\h}{\calL^{2} \calF^k\h + \calF^{k+1}\h}\right). \]
	\begin{itemize}
		\item[$\centerdot$] If $\vartheta_k^1$  is non-zero, then so is $\vartheta_k$.
		\item[$\centerdot$] If $\vartheta_{k}^1 = 0$~, there exists $\upsilon_1 \in \mathfrak{h}_0$~, such that $\upsilon_1 \cdot \xi_1 \in \calL^{2} \calF^k \h + \calF^{k+1}\h$, by the implication $(1) \Rightarrow (2)$ of  \cref{prop:obstru}.
	\end{itemize}	
	\noindent The construction of such obstruction classes can be performed higher up in a similar way. This leads to a sequence of classes $\left(\vartheta_k^i \right)_{1 \leqslant i \leqslant \eta}$ which is either 
	\begin{itemize}
		\item[$\centerdot$] an infinite sequence of vanishing homology classes, when $\eta_k= \infty$~, or
		\item[$\centerdot$]  a finite sequence of trivial classes that ends on a nonzero class $\vartheta_{\eta}$~, when $\eta_k \in \mathbb{N}$~.
	\end{itemize}  
\end{construction}

\begin{remark}\label{bound}
    By Assumptions \ref{ass:gradings}, the filtration $\calL$ is bounded and there exists $\delta_{k+1}$ such that \[\calL^{\delta_{k+1}}\h \subseteq \calF^{k+1} \h \ .\] In the case where $\eta \in \mathbb{N}$ is finite, we have $\eta_k < \delta_{k+1} - 1$. In the case where $\eta_k = \infty$, the construction starts to be trivial at the level $k = \delta_{k+1} - 1$ where $\xi_k \in \calF^{k+1} \h $. 
\end{remark}

\begin{definition}\label{gauge triviality sequ}
	Let $\xi \in \mathrm{MC}(\h)$ be a Maurer--Cartan element such that $\xi \in \calF^{k} \h$ for $k \geqslant 2$. A \emph{$k^{th}$-intermediate gauge triviality sequence} of $\xi$ is an obstruction sequence \[\left(\vartheta_k^i \right)_{0 \leqslant i \leqslant \eta_k}   , \quad  \eta_k \in \lbrace 0 , \dots , \delta_{k+1}, \infty \rbrace \ ,  \] obtained through \cref{constru}. 
\end{definition}

\begin{lemma}\label{gauge deg}
    Let $\xi \in \mathrm{MC}(\calF^{k} \h)$ for $k \geqslant 2$. The index $\eta_k \in \lbrace 0 , \dots , \delta_{k+1}, \infty \rbrace $ of the last class of a $k^{th}$-intermediate gauge triviality sequence only depends on $\xi$, i.e. given \[ \left(\vartheta_k^i \right)_{0 \leqslant i \leqslant \eta_k} \quad \and  \quad  \left( \vartheta_k^{'i}\right)_{0 \leqslant i \leqslant \eta_k'}  \] two $k^{th}$-intermediate gauge triviality sequences, we have $\eta_k = \eta_k'$. This element is called the \emph{$k$-th intermediate gauge triviality degree} of $\xi$.
\end{lemma}

\begin{proof}
    The proof is the same than the one of \cite[Lemma~1.9]{CE24b}.
\end{proof}

\begin{theorem}\label{prop:bifilteredTriviality}
     Let $\xi \in \mathrm{MC}(\calF^{k} \h)$ for $k \geqslant 2$. The following assertions are equivalent. 
     \begin{enumerate}
         \item The $k$-th intermediate gauge triviality degree of $\xi$ is equal to $\infty$;
         \item The obstruction $ \vartheta_k =  \left[ \pi_{k+2} (\phi_{k+1})\right] \in H_{-1}\left( \mathfrak{h} / \mathcal{F}^{k+2} \mathfrak{h}  \right) , $ vanishes;
         \item There exists $\omega \in \h_0$ such that $\omega \cdot \xi \in \calF^{k+1} \h$. 
     \end{enumerate}
\end{theorem}

\begin{proof}
    Let us prove the implication $(1) \Rightarrow (3)$. Let $\left(\vartheta_k^i \right)_{0 \leqslant i \leqslant \eta_k} $ be a $k$-th intermediate gauge triviality sequence of $\xi$ and let us denote by $\left(\upsilon_k \right)$ and $\left(\xi_k \right)$ the associated sequence of gauges and Maurer--Cartan elements given by \cref{constru}.  If $\eta_k = \infty$~, it follows from the construction that the gauge  \[\omega \coloneqq \mathrm{BCH}(\upsilon_{\delta_{k+1} -1}, \mathrm{BCH}(\cdots \mathrm{BCH}(\upsilon_{2}, \upsilon_{1})) \cdots ) \ , \] satisfies $\omega \cdot \varphi \in \mathcal{F}^{k+1} \h$, see \cref{bound}. The equivalence $(2) \Leftrightarrow (3) $ is given by \cite[Proposition~1.6]{CE24b}. The implication $(2) \Rightarrow (1)$ follows from \cref{vanishing}. Indeed, if the $k$-th intermediate gauge triviality degree $\eta_k$ of $\xi$ is not equal to $\infty$ and \[\vartheta^{\eta_k}_k \neq 0 \implies \vartheta_k \neq 0 \ . \qedhere\]  
\end{proof}

\begin{lemma}\label{prop:dependance}
    Let $\phi \in \mathrm{MC}( \calF^2\h)$ and suppose that there exist $\omega, \omega' \in \mathfrak{h}_0$  such that $\omega \cdot \phi$ and $\omega' \cdot \phi$ are both in $\calF^k \h$. Then, they have equal $k$-th intermediate gauge triviality degrees.
\end{lemma}

\begin{proof}
    Suppose by contradiction that the $k$-th intermediate gauge triviality degree of $\omega \cdot \phi$ is $\eta_k \in \{0,1,\dots,\infty\}$, and that of $\omega \cdot \phi$ is a finite number such that $\eta_k' < \eta_k$.  By construction, there exists $\lambda$ such that $\lambda \cdot (\omega \cdot \phi) \in \calL^{\eta_k'+1} \calF^k \h + \calF^{k+1} \h$. But we can rewrite this element as 
    \[ \lambda \cdot (\omega \cdot \phi) = \mathrm{BCH}(\lambda,\omega)\cdot \phi = \mathrm{BCH}(\mathrm{BCH}(\lambda,\omega), - \omega') \cdot (\omega' \cdot \phi) \ . \]
    By \cref{prop:obstru}, the $k$-th intermediate gauge triviality degree of $\omega' \cdot \phi$ must be strictly larger than $\eta_k'$, which is impossible. The case $\eta_k < \eta_k'$ is treated analogously.
\end{proof}

\begin{definition}
    Let $\phi \in \mathrm{MC}( \calF^2\h)$ be a Maurer--Cartan element and suppose that there exist $\omega \in \mathfrak{h}_0$ such that $\omega \cdot \phi \in \mathcal{F}^k \h$. The \emph{$k$-th intermediate gauge triviality degree} $\phi$ is defined and is the one of $\omega \cdot \phi$. 
\end{definition}

\begin{theorem}\label{thm:gaugeTriviality}
    Under Assumptions \ref{ass:gradings}, the following assertions are equivalent. 
    \begin{enumerate}
        \item The Maurer--Cartan element $\varphi$ is gauge equivalent to its first component $\varphi^{(1)}$. 
        \item The gauge triviality degree of $\phi \coloneqq \varphi - \varphi^{(1)}$ is equal to $\infty$.
        \item For all $k \geqslant 2$, the $k^{th}$-intermediate gauge triviality degree of $\phi$ is defined, and equal to $\infty$. 
    \end{enumerate}   
\end{theorem}

\begin{proof}
    The equivalence $(1) \Leftrightarrow (2)$ is given by \cite[Theorem~1.17]{CE24b}. Let us prove the equivalence $(2) \Leftrightarrow (3)$. The $k^{th}$-intermediate gauge triviality degree is defined for all $k \geqslant 2$ if and only if there exists $\omega_k \in \mathfrak{h}_0$ such that $\omega_k \cdot \phi \in \mathcal{F}^k\mathfrak{h}$. By \cite[Theorem~1.11]{CE24b}, this is equivalent to the gauge triviality degree of $\phi$ being infinite. Now, by \cref{prop:bifilteredTriviality}, for each $k \geqslant 1$, the $(k+1)^{th}$-intermediate gauge triviality degree is defined if and only if the $k^{th}$-intermediate gauge triviality degree is defined and infinite. 
 \end{proof}

\subsection{Rigidity criteria}
Let us recall the notion of rigidity of a Maurer--Cartan element.
\begin{definition}\label{def:rigid}
    A Maurer--Cartan element $\psi$ concentrated in weight one is \emph{rigid} when any of its deformations is gauge-trivial, that is, when any Maurer--Cartan element $\varphi$ satisfying $\varphi^{(1)} =  \psi$ is gauge equivalent to $\psi$. 
\end{definition}

\noindent The obstruction theory of Maurer--Cartan solutions gives the following characterization.
\begin{proposition}[{\cite[Corollary~1.18]{CE24b}}]
    Let $\mathfrak{g}$ be a weight-graded dg Lie algebra, and $\psi$ a Maurer--Cartan element $\psi$ concentrated in weight one. If
    \[\mathcal{F}^{2} H_{-1}(\mathfrak{g}^{\psi}) = 0 \]
    then $\psi$ is automatically rigid.
\end{proposition}

\noindent In the presence of an extra filtration $\calL$, we can refine the rigidity criterion above. We first consider the following notion of rigidity with respect to the extra filtration.
\begin{definition}\label{def:Lrigidity}
    The element $\psi$ is \emph{$\calL^m$-rigid} for some $m \geqslant 0$ if every $\phi \in \mathrm{MC}(\g^\psi)$ satisfying $\phi \in \calF^2\g^\psi$ is gauge-equivalent to some element in $\calL^{m+1} \g^\psi$.
\end{definition}

\begin{proposition}\label{prop:intermediateToTotalRigid}
    Under Assumptions \ref{ass:gradings}, if $\psi$ is $\calL^0$-rigid, and if the image of the map 
    \[ \calF^2 H_{-1}(\gr_\mathcal{L}^i \g^\psi) \longrightarrow H_{-1}(\g^\psi/\mathcal{L}^{i+1}\g^\psi)\] 
    vanishes for every $i \geqslant 1$, then $\psi$ is rigid.
\end{proposition}
\begin{proof}
    Let us consider some Maurer--Cartan element $\varphi$ such that $\varphi^{(1)} = \psi$, and again denote $\h =\g^\psi$. Since $\phi = \varphi - \psi \in \mathrm{MC}(\h)$ is in $\calF^2\h$ and $\psi$ is $\calL^0$-rigid, we can assume $\phi \in \calL^i\calF^k\h$, for some $i \geqslant 1$ and $k \geqslant 2$. The Maurer--Cartan equation then implies that
    \[ d^\psi \phi = -\frac{1}{2}[\phi,\phi] \equiv 0 \pmod{\calL^{i+1} \h}, \]
    so $[\phi]$ is a class in $\calF^k H_{-1}(\gr_\calL^i \h)$. We now consider the maps
    \small \[  \calF^k H_{-1}\left(\gr_\calL^i \h\right) \to \calF^k H_{-1}\left(\frac{\h}{\mathcal{L}^{i+1}\h}\right) \to \calF^k H_{-1}\left(\frac{\h}{\mathcal{L}^{i+1}\h+ \calF^{k+1}\h}\right) \xleftarrow{f} \calF^k H_{-1}\left(\frac{\h}{\mathcal{L}^{i+1}\calF^k\h+ \calF^{k+1}\h}\right).  \]
    \normalsize The last group is where the $k$-th intermediate gauge triviality class $\vartheta_{k}^i$ lives. But the image of this class in $\calF^k H_{-1}(\h/\mathcal{L}^{i+1}\h+ \calF^{k+1}\h)$ is equal to the image of $[\phi] \in \calF^k H_{-1}(\gr_\calL^i \h)$, which vanishes by assumption. Therefore, to show that all classes $\vartheta_k^i$ vanish, it is sufficient to show that the map $f$ is injective. Suppose we have $[\alpha] \in \ker(f)$. Then we can find $\lambda$ such that 
    \[ d^\psi \lambda \equiv \alpha \pmod{\calL^{i+1}\h + \calF^{k+1}\h}. \] Let us set $\lambda' = \lambda^{( k-1)} + \lambda^{( k)} + \cdots $ Since $d^\psi$ is homogeneous of weight one,  we must have
    \[ d^\psi \left( \lambda' \right) \equiv \alpha \pmod{\calL^{i+1}\h + \calF^{k+1}\h}, \]
    since $\alpha$ is in $\calF^k\h$, in other words, $\lambda'$ is a primitive of $\alpha$ in $\h/\mathcal{L}^{i+1}\calF^k\h+ \calF^{k+1}\h$.
\end{proof}

\subsection{Application to formality of properadic algebras}
\noindent 
Let us recall the approach of formality as a deformation problem, following \cite{PHC,Kaledin}. Let $\C$ be a reduced weight-graded dg coproperad, and as in \cref{def:coformality} we have the canonical morphism
\[ \Omega\calC \to \calP = \left. \calT(s^{-1}\calC^{(1)}) \middle/ d_{\Omega\calC}(s^{-1}\calC^{(2)}) \right. \]
Any chain complex $A$ and $\Omega\calC$-algebra structure $\varphi$ on $A$, its homology $H = H_*(A)$ has two $\Omega\calC$-algebra structures:
\begin{itemize}
    \item an \emph{induced} $\calP$-algebra structure $\varphi_*$, in other words, a $\Omega\calC$-algebra structure concentrated in weight one, and
    \item  a \emph{transferred} $\Omega\calC$-algebra structure $\varphi_t$ whose weight one component agrees with $\varphi_*$ and such that $(A,\varphi)$ and $(H,\varphi_t)$ are quasi-isomorphic as $\Omega\calC$-algebras.
\end{itemize}
As a consequence, the $\Omega\calC$-algebra $(A,\varphi)$ is formal if and only if there is a zig-zag of $\infty$-quasi-isomorphisms linking $(H,\varphi_*)$ and $(H,\varphi_t)$. By the invertibility of $\infty$-quasi-isomorphisms \cite[Theorem~4.18]{PHC}, this is equivalent to the existence of a single $\infty$-quasi-isomorphism, which, since $H$ has vanishing differential, is given by a $\infty$-isotopy up to an isomorphism of $H$.

An $\infty$-isotopy is an $\infty$-morphism of $\Omega\calC$-algebra structures whose first component is the identity. Recall the convolution dg Lie algebra
\[ \mathfrak{g}_H =\left(\Hom_{\mathbb{S}}\left(\overline{\C}, \End_H\right), \partial, \star\right), \] 
whose Maurer--Cartan elements are in bijection with $\Omega \calC$-algebra structures on $H$. By \cite[Proposition~1.1]{campos2025universal}, $\varphi_t$ and $\varphi_*$ are $\infty$-isotopic if and only if the corresponding Maurer--Cartan elements in $\mathfrak{g}_H$, also denoted $\varphi_*$ and $\varphi_t$, are gauge-equivalent. Putting together all of these results, we have the following characterization of formality and intrinsic formality.
\begin{proposition}\label{prop:equivalenceFormality}
    A $\Omega\calC$-algebra $(A,\varphi)$ is formal if and only if $\varphi_*$ and $\varphi_t$ are gauge equivalent. A $\calP$-algebra $(H,\psi)$ is intrinsically formal if and only if $\psi$ is rigid, in the sense of \cref{def:rigid}.
\end{proposition}

\noindent Suppose now that $\calC$ is a reduced weight-graded coproperad (with no differential) and let $H$ be a graded $R$-module. Then, the associated convolution dg Lie algebra $\g_{H}$ is a weight-graded Lie algebra in the sense of \cref{ass:gradings}(1), with the weight grading coming from that of $\C$, and associated filtration $\calF^N \g_H$. The dg Lie algebra $\mathfrak{g}_H$ always has an extra filtration where $\calL^i\g_H$ is all the operations with $(i+1)$ or more inputs. More precisely, the $(\S^{op}\times \S)$-module $\calC$ has a direct sum decomposition under a \emph{second} grading
\[ \calC = I \oplus \calC_{(1)} \oplus \calC_{(2)} \oplus \calC_{(3)} \oplus  \dots \]
where $\calC_{(\ell)}$ is spanned by the operations with $\ell$ outputs. This gives a decomposition
\[ \g_H = (\g_H)_{(1)} \times (\g_H)_{(2)} \times (\g_H)_{(3)} \times \dots \]
and we can then define the extra filtration by
\[ \calL^i\g_H = \prod_{\ell \geqslant i+1} (\g_H)_{(\ell)}. \]

\begin{theorem}
    Let $\C$ be a reduced weight-graded coproperad, such that the extra filtration $\calL$ on $\g_{H}$ is bounded with respect to $\calF$, satisfying (2) of \cref{ass:gradings}. Suppose that $(H, \gamma)$ be an $\Omega \C$-algebra structure on a graded $R$-module $H$. The following assertions are equivalent:
    \begin{enumerate}
        \item There exists an $\infty$-quasi-isomorphism between $(H, \gamma)$ and $\left(H, \gamma^{(1)} \right)$ 
        \item The Maurer--Cartan elements $\gamma$ and $\gamma^{(1)}$ are gauge equivalent in $\g_H$. 
        \item The gauge-triviality degree of $\phi = \gamma - \gamma^{(1)}$ is equal to $\infty$.
        \item For all $k \geqslant 1$, the $k$-th intermediate gauge triviality degree of $\phi$ is defined and equal to $\infty$.
    \end{enumerate}
\end{theorem}
\begin{proof}
    The equivalence $(1) \Leftrightarrow (2)$ is \cite[Proposition~1.1]{campos2025universal}, and the remaining equivalences are given by \cref{thm:gaugeTriviality}.
\end{proof}

\noindent When $\gamma$ and $\gamma^{(1)}$ are respectively the transferred and induced structures $\varphi_t$ and $\varphi_*$ on the homology of a dg $\Omega\calC$-algebra $(A,\varphi)$, one can then apply the theorem above to characterize the formality of $\varphi$. Similarly, we can use the rigidity criteria developed in the previous subsection to characterize intrinsic formality. Combining \cref{prop:intermediateToTotalRigid} and \cref{prop:equivalenceFormality}, we have the following result.
\begin{theorem}\label{thm:criterion}
    Let $\C$ be a reduced weight-graded coproperad, such that the extra filtration $\calL$ on $\g_{H}$ is bounded with respect to $\calF$, satisfying (2) of \cref{ass:gradings}, and $(H, \psi)$ be an $\Omega \C$-algebra structure concentrated in weight one (that is, a $\calP$-algebra structure), which is $\calL^0$-rigid. If the image of the map
    \[ \calF^2 H_{-1}(\gr_\mathcal{L}^i \g^{\psi}) \longrightarrow H_{-1}(\g^{\psi}/\mathcal{L}^{i+1}\g^{\psi}) \] 
    vanishes for every $i \geqslant 1$, then $(H,\psi)$ is intrinsically formal.
\end{theorem}

\section{\textcolor{bordeau}{Coformality of based loop spaces of spheres}}\label{sec:formalityLoopSpaces}
\noindent In this section, we use the formalism explained in the previous sections to study formality of $\calY^{(n)}_\infty$-algebra structures on the homology algebra of based loop spaces of spheres. That is, we consider $\Omega \calC$-structures where $\calC = \calY^{(n) \antish}$; for any chain complex $(A,d_A)$, the corresponding convolution dg Lie algebra $\g_A$ is given by
\begin{align*}
    \g_A &= \prod_{i \geqslant 2} \Hom(A[1]^{\otimes i},A)[1] \\
    &\times \prod_{ij\neq 0} \Hom(A[1]^{\otimes i} \otimes A[1]^{\otimes j}, A\otimes A)^{(2,n)}[n-1] \\
    &\times \prod_{\ell \geqslant 3} CH^*_{(\ell)}(A)^{(\ell,n)}[(n-2)(\ell-1)+1] \ .
\end{align*}
Recall that the superscript $(-)^{(\ell,n)}$ denotes the subcomplex of elements that have the appropriate cyclic symmetry. The dg Lie algebra $\g_A$ is the dg Lie subalgebra of $CH^*_{[n]}(A)[1]$ of operations with total arity (number of in- and outputs) greater than or equal to $3$, endowed with the necklace bracket, which we will just denote by $[-,-]$. The weight grading is given by the \emph{total number of legs} minus two. That is, the weight of a map $\varphi: A^{\otimes N_\mathrm{in}} \to A^{\otimes N_\mathrm{out}}$ is
\[ \wt(\varphi) = N_\mathrm{in} + N_\mathrm{out} - 2 \ . \] 
The extra filtration $\calL$ is bounded with respect to $\calF$, since $\calL^{i}$ is spanned by all operations with at least $(i+1)$-outputs, so for any $k$ we have 
\[ \calL^{k+1} \g_A \subseteq \calF^k \g_A, \]
and $\calL$ satisfies (2) of \cref{ass:gradings}. We note that the differential on every graded piece $\gr_\mathcal{L}^{\ell-1} \g^{\psi}_A$ is given by taking necklace bracket with the $A_\infty$-structure $\mu$, which is the usual differential on $CH^*_{(\ell)}(A)$. In other words, unraveling all the degree shifts, we find that for every $\ell \geqslant 3$ we have an isomorphism of complexes
\[ \gr_\mathcal{L}^{\ell-1} \g^{\psi}_A \cong CH^{-*+(n-2)(\ell-1)+1}_{(\ell)}(A)^{(\ell,n)} \]
and for $\ell = 2$ we have an injection
\[ \gr_\mathcal{L}^{\ell-1} \g^{\psi}_A \hookrightarrow CH^{-*+(n-2)(\ell-1)+1}_{(\ell)}(A)^{(\ell,n)} \]
given by the inclusion of elements of total arity $\geqslant 3$. 

\medskip
\noindent Note that $\calL^0\g_A/\calL^1\g_A$ is nothing but the convolution algebra governing $A_\infty$-algebras on $(A,d_A)$, since it has exactly the operations with one output.
\begin{proposition}\label{prop:AinftyRigidity}
    Let $(H,\psi)$ be a graded $\calY^{(n)}$-algebra, and $(H,\psi_{(1)})$ the corresponding associative algebra. If $(H,\psi_{(1)})$ is intrinsically formal as an associative algebra, then $(H,\psi)$ is $\calL^0$-rigid, in the sense of \cref{def:Lrigidity}.
\end{proposition}
\begin{proof}
    Let $\psi + \phi$ be a deformation of $\psi$ with $\phi \in \calF^2\g^\psi$. We decompose
    \[ \phi = \phi_{(1)} + \phi_{(2)} + \dots \]
    where $\phi_{(\ell)}$ has $\ell$ outputs. Note that $(\psi_{(1)} + \phi_{(1)})$ must be an $A_\infty$-structure. If $\psi_{(1)}$ is intrinsically formal as an associative algebra, then there is a gauge transformation trivializing $\phi_1$; so applying it to all of $\phi$ we get a gauge transformation sending it into an element in $\calL^2\g_H$.
\end{proof}

\medskip

\noindent We now turn to the specific example of based loop spaces of spheres. Recall the definition of coformality for a local Poincar\'e duality pair $(X,[X])$, given in \cref{def:coformality}.

\begin{theorem}\label{thm:sphereIntrinsicCoformality}
     For any $n \geqslant 1$, the pair $(S^n,[S^n])$ is intrinsically coformal.
\end{theorem}
\begin{proof}
    Let us first present the case $n \geqslant 2$. The relevant homology algebra is 
    \[H = H_*(\Omega S^n) \cong R[t] \ ,\]
    where $t$ has degree $(n-1)$. This is a smooth $n$-Calabi--Yau algebra, with weak CY structure represented by $\omega  = 1[t] \in CH_n(H)$ and compatible $n$-pre-CY structure given by
    \[ \psi = \mu + \alpha, \quad \mu \in CH^2(H), \quad \alpha \in CH^n_{(2)}(H), \]
    where $\mu$ is the usual multiplication on $H$, see \cite[Proposition~7.1]{rivera2023algebraic}. 
    Since $\psi$ is a $\calY^{(n)}$-algebra structure, it has homogeneous weight one in $\g_H$, and in terms of the extra filtration $\mathcal{L}$ we have 
    \[\mu \in \calL^0\g_H \quad \mbox{and} \quad \alpha \in \mathcal{L}^1\g_H \ .\]
    First we note that $H$ is an intrinsically formal associative algebra, so due to \cref{prop:AinftyRigidity} $\psi$ is $\calL^0$-rigid. Therefore, by \cref{thm:criterion}, we just have to prove that the image of the maps 
    \[ \calF^2 H_{-1}(\gr_\mathcal{L}^i \g_H^\psi) \to H_{-1}(\g_H^\psi /\mathcal{L}^{i+1}\g_H^\psi) \]
    vanish for each $i \geqslant 1$. The graded pieces of the $\mathcal{L}$ filtration, with the induced differential, give subcomplexes
    \[  \calF^2 \gr_\mathcal{L}^{\ell-1} \g_H^\psi \subset \left(CH^{* + n\ell -n -2\ell+3}_{(\ell)}(H), [\mu,-] \right) \]
    spanned by the vertices that are cyclically (anti)symmetric and have 4 total legs or more. Note that here we changed the index using $\ell = i+1$ so that $\ell$ denotes the number of output legs. In order to understand the cohomology of the right-hand side of the inclusion above, we use the fact that the elements $\alpha$ and $\omega$ are inverses of each other, in the sense of \cite{KTV2}; they give chain equivalences between $CH_*(H)$ and all the complexes $CH^*_{(\ell)}(H)$, up to appropriate shifts. Let us be more precise: for any $\ell \geqslant 1$, there is an explicit quasi-isomorphism
    \[ g_{(\ell)} \colon CH_*(H) \xrightarrow{\sim} CH^{-*+n\ell}_{(\ell)}(H) \]
    given by evaluating a certain diagram on some Hochschild chain $x$. For example, when $\ell=3$ this diagram is
    \[\tikzfig{
	\node [vertex] (x) at (0,0) {$x$};
	\node [bullet] (se) at (0.5,-0.87) {};
	\node [bullet] (ne) at (0.5,0.87) {};
	\node [bullet] (w) at (-1,0) {};
	\node [bullet] (t) at (-0.87,0.5) {};
	\draw [->-] (x) to (t);
	\draw [->-=1] (ne) to (1,1.74);
	\draw [->-=1] (se) to (1,-1.74);
	\draw [->-=1] (w) to (-2,0);
	\draw (1,0) arc (0:60:1);
	\draw (1,0) arc (0:-60:1);
	\draw [->-] (-0.5,0.87) arc (120:60:1);
	\draw [->-=0.8] (-0.5,0.87) arc (120:150:1);
	\draw [->-=0.8] (-0.87,0.5) arc (150:180:1);
	\draw [->-] (-0.5,-0.87) arc (-120:-60:1);
	\draw [->-] (-0.5,-0.87) arc (-120:-180:1);
	\node [vertex,fill=white] (e) at (1,0) {$\alpha$};
	\node [vertex,fill=white] (nw) at (-0.5,0.87) {$\alpha$};
	\node [vertex,fill=white] (sw) at (-0.5,-0.87) {$\alpha$};
        \node at (-2.3,0) {$1$};
        \node at (1.2,-2) {$2$};
        \node at (1.2,2) {$3$};
    }\]
    The diagrams for other values are similar, with $\ell$ copies of the $\alpha$ vertex around the circle. We now recall the Hochschild homology of $H$, in other words the homology of the free loop space $\Lambda S^n$, together with a choice of representative for the generator of each factor. When $n$ is odd, we have:
    \[
        \begin{array}{cccccccccccccccccccc}
            HH_*(H) & \cong & R & \oplus & R[n-1] & \oplus & R[n] & \oplus & R[2n-2] & \oplus & R[2n-1] & \oplus \cdots\\
             & & \rotatebox[origin=c]{90}{$\in$} & & \rotatebox[origin=c]{90}{$\in$} & & \rotatebox[origin=c]{90}{$\in$} & & \rotatebox[origin=c]{90}{$\in$} & & \rotatebox[origin=c]{90}{$\in$} \\
             & & 1 & & t & & 1[t] & & t^2 & & t[t] & 
        \end{array}
    \]
    and when $n$ is even, we have
    \[
        \begin{array}{cccccccccccccccccccc}
            HH_*(H) & \cong & R & \oplus & R[n-1] & \oplus & R[n] & \oplus & R[3n-3] & \oplus & R[3n-2] & \oplus \cdots\\
             & & \rotatebox[origin=c]{90}{$\in$} & & \rotatebox[origin=c]{90}{$\in$} & & \rotatebox[origin=c]{90}{$\in$} & & \rotatebox[origin=c]{90}{$\in$} & & \rotatebox[origin=c]{90}{$\in$} \\
             & & 1 & & t & & 1[t] & & t^3 & & t^2[t] & 
        \end{array}
    \]
    Let us study the cyclic symmetry properties of the elements in $CH^*_{(\ell)}(H)$ that are images of the representatives above. Let us denote by
    \[ \rho \colon CH^*_{(\ell)}(H) \to CH^*_{(\ell)}(H) \]
    the rotation automorphism; for an element $\xi$ to have the appropriate symmetry we must have $\rho(\xi) = (-1)^{(n-1)(\ell-1)} \xi$. Using a similar diagram to the above, and using the sign calculus explained in \cite{KTV1}, we find a homotopy between 
    \[  g_{(\ell)} \quad \text{and} \quad  (-1)^{(n-1)(\ell-1)} \rho \circ g_{(\ell)}. \] 
    Since $R$ is a $\QQ$-algebra, we can symmetrize, and for any $k$ and any closed element $x$ of $CH_*(H)$ we use the homotopy above to conclude that there is a homology equivalence
    \[ g_{(\ell)}(x) \simeq \frac{1}{\ell}\left( \sum_{i=0}^{\ell-1}  (-1)^{i(n-1)(\ell-1)} \rho^i (g_{(\ell)}(x)) \right), \]
    and therefore all the classes $[g_{(\ell)}(x)]$ can be represented by appropriately symmetric elements. Using the explicit formula for $\alpha$ and for the map $g_{(\ell)}$, we calculate that, for any $\ell$ and $k$,
    \begin{enumerate}
        \item the element $g_{(\ell)}(1[t^{k+1}])$ is already appropriately symmetric on the nose, and cohomologous to  $(k+1) g_{(\ell)}(t^k[t])$, since $1[t^{k+1}]$ is homologous to $(k+1) t^k[t]$, and 
        \item the following equation holds
        \[ \left[\alpha, \frac{1}{\ell-1}\left( \sum_{i=0}^{\ell-2} \rho^i (-1)^{i(n-1)(\ell-2)}(g_{(\ell-1)}(t^{k+1})) \right)\right] = g_{(\ell)}(1[t^{k+1}]). \]
    \end{enumerate}
    By counting degrees, we know that any homology class of $\gr_\mathcal{L}^{\ell-1} \g_H^\psi$ in degree $-1$ must be the image of a class of degree $(n+2\ell-4)$ in $HH_*(H)$ under the map $g_{(\ell)}$, thus a class whose homological degree has the same parity as $n$. These are exactly the classes represented by a multiple of $t^k[t]$ for some $k$. But the differential on $\g^\psi_H$ is given by taking commutator with $\psi = \mu + \alpha$, so $(1)$ and $(2)$ above imply that all these classes $g_{(\ell)}(t^k[t])$ vanish in the relevant group
    \[ H_{-1}\left(\g_H^\psi/\mathcal{L}^{\ell}\g_H^\psi\right),\] 
    and so we can apply \cref{thm:criterion}.

    \medskip
    
    \noindent In the remaining case $n=1$, where $H=k[t,t^{-1}]$, with $\deg(t)=0$. Here the compatible $\calY^{(n)}$-algebra structures are of the form $\psi = \mu + \alpha + \tau$,
    where $\tau$ is some element with zero inputs and three outputs, see for example the structure in \cite[Section~7.4]{rivera2023algebraic}. There will be many compatible $\calY^{(n)}$-algebra structures, as the dimension condition in \cref{prop:compatibleUniqueness} fails. Nevertheless, the same argument we presented above in the $n\ge 2$ case applies to any such structure, but even more simply. Since $HH_*(H)$ is concentrated in degrees zero and one when $n=1$, there can only be nontrivial classes of degree $(n+2\ell-4)$ in $HH_*(H)$ when $\ell = 2$. So the only classes in homological degree $-1$ in $\gr_\mathcal{L}^{\ell-1} \g_H^\psi$ are multiples of the classes $g_{(2)}(t^k[t])$, but all of these are weight one. We conclude that $\calF^2 H_{-1}(\gr_\mathcal{L}^{\ell-1} \g_H^\psi)$ vanishes for all $\ell \ge 2$ and we can again apply \cref{thm:criterion}.
\end{proof}

\bibliographystyle{abbrv}
\bibliography{bib}

\end{document}